\documentclass[a4paper,11pt]{article}

\usepackage{amsmath, amssymb, amsthm}
\usepackage{graphics}

\usepackage{psfrag, booktabs}
\usepackage{color, epsfig, float}
\usepackage[notcite,notref,final]{showkeys}
\usepackage[active]{srcltx}
\usepackage[bookmarksopen, colorlinks, linkcolor = blue, urlcolor = red,
            citecolor = red, menucolor = blue]{hyperref}

\usepackage{mathtools}

\newtheorem{theorem}{Theorem}[section]
\newtheorem{lemma}[theorem]{Lemma}

\newtheorem{proposition}[theorem]{Proposition}
\newtheorem{remark}[theorem]{Remark}

\DeclareMathOperator{\argmin}{arg\, min}

\newcommand\R{\mathbb{R}}

\newcommand\N{\mathbb{N}}

\newcommand{\inner}[2]{\langle #1, #2\rangle}

\newcommand\norm[1]{\|#1\|}
\newcommand\Norm[1]{\left\|#1\right\|}
\newcommand\set[1]{\{#1\}}

\begin{document}

%
\title{On a family of gradient type projection methods for nonlinear ill-posed problems}
\setcounter{footnote}{1}

\author{
A.~Leit\~ao%
\thanks{Department of Mathematics, Federal Univ.\,of St.\,Catarina, P.O.\,Box 476,
        88040-900 Florian\'opolis, Brazil.
        \href{mailto:acgleitao@gmail.com}{\tt acgleitao@gmail.com}.}
\and
B.~F.~Svaiter%
\thanks{IMPA, Estr.\,Dona Castorina 110, 22460-320 Rio de Janeiro, Brazil.
        \href{mailto:benar@impa.br}{\tt benar@impa.br}.} }
\date{\small \today}

\maketitle

\begin{abstract}
We propose and analyze a family of successive projection methods whose
step direction is the same as Landweber method for solving nonlinear
ill-posed problems that satisfy the {\em Tangential Cone Condition} (TCC).
This family enconpasses Landweber method, the minimal error method, and
the steepest descent method; thush providing an unified framework for
the analysis of these methods.
Moreover, we define in this family new methods which are convergent for
the constant of the TCC in a range {\em twice as large} as the one required
for the Landweber and other gradient type methods.

The TCC is widely used in the analysis of iterative methods for solving
nonlinear ill-posed problems.
The key idea in this work is to use the TCC in order to construct
special convex sets possessing a separation property, and
to succesively project onto these sets.

Numerical experiments are presented for a nonlinear 2D elliptic parameter
identification problem, validating the efficiency of our method.
\end{abstract}

\medskip


\section{Introduction} \label{sec:intro}

In this article we propose a family of succesive orthogonal projection methods
for obtaining stable approximate solutions to nonlinear ill-posed operator
equations.

The \textit{inverse problems} we are interested in consist of determining an
unknown quantity $x \in X$ from the data set $y \in Y$, where $X$, $Y$ are
Hilbert spaces. The problem data $y$ are obtained by indirect measurements of
the parameter $x$, this process being described by the model $F(x) = y$, where
$F: D \subset X \to Y$ is a non-linear ill-posed operator with domain $D = D(F)$.

In practical situations, the exact data $y$ is not known. Instead, what is
available is only approximate measured data $y^\delta \in Y$ satisfying
\begin{equation}\label{eq:noisy-i}
    \norm{ y^\delta - y } \le \delta \, ,
\end{equation}
where $\delta > 0$ is the noise level.
Thus, the abstract formulation of the inverse problems under consideration is
to find $x \in D$ such that
\begin{equation}\label{eq:inv-probl}
    F(x) \ = \ y^\delta \, .
\end{equation}

Standard methods for obtaining stable solutions of the operator equation
in \eqref{eq:inv-probl} can be divided in two major groups, namely, 
 \textit{Iterative type}
regularization methods \cite{BK04, EHN96, HNS95, KNS08,
Lan51} and \textit{Tikhonov type} regularization methods \cite{EHN96,
Mor93, SV89, Tik63b, TA77, Sch93a}.
A classical and general condition commonly used in the convergence analysis
of these methods is the \emph{Tangent Cone Condition} (TCC)~\cite{HNS95}.

In this work we use the TCC to define convex sets containing the local solutions
of \eqref{eq:inv-probl} and devise a family of succesive projection methods.
%
%
The use of projection methods for solving linear ill-posed problems dates back to the
70's (with the seminal works of Frank Natterer and Gabor Herman)
\cite{Nat77, Nat86, Her75, Her80}. The
\emph{combination} of Landweber iterations with projections onto a feasible set,
for solving \eqref{eq:inv-probl} with $y^\delta$ in a convex set was analyzed in
\cite{Eik92} (see also \cite{EHN96} and the references therein).

The distinctive features of the family of methods prposed in this work are as follows:
\begin{itemize}
\item[$\bullet$] the basic method in this family outperformed, in our preliminary
numerical experiments, the classical Landweber iteration \cite{HNS95} as well as
the steepest descent iteration \cite{NS95}
(with respect to both the computational cost and the number of iterations);
\item[$\bullet$] the family is generated by introducing relaxation in the
stepsize of the basic method and such a family encompasses, as particular cases,
the Landweber method, the steepest descent method, as well as the minimal error method
\cite{NS95};
thus, providing an unified framework for their convergence analysis;
\item[$\bullet$] the basic method within the family converges for the constant in the TCC twice
as large as required for the convergence of the Landweber and other gradient
type methods.
\end{itemize}
In view of theses features, the basic method within the proposed family
is called Projected Landweber (PLW) method. Although in the linear case the
PLW method coincides with the minimal error method, in the nonlinear case
these two methods are distinct.

%

{\color{black}
Landweber iteration was originally proposed for solving linear
equations by using the method of successive approximations  applied to
the normal equations~\cite{Lan51}.
Its extension to non-linear equations was obtained substituting the adjoint
of the linear map 
by the Jacobian's adjoint of the operator under consideration~\cite{HNS95}.
Such a method is named (nonlinear) Landweber, in
the setting of ill-posed problems.}
Convergence of this method in the nonlinear case under the TCC was proven
by Hanke et al.~\cite{HNS95}.
Convergence analysis for the steepest descent method and minimal error method
(in the nonlinear case) can be found in \cite{NS95}.


Although Levenberg-Marquardt type methods are faster than gradient type methods,
with respect to the number of iterations, gradient type methods have simpler
and faster iteration formulas. Moreover, they fit nicely in Cimino and Kaczmarz
type schemes. For these reasons, acceleration of gradient type methods is a
relevant topic in the field of ill-posed problems.
\bigskip

The article is outlined as follows.
In Section~\ref{sec:nbr} we state the main assumptions and derive some auxiliary
estimates required for the analysis of the proposed family of methods.
In Section~\ref{sec:ple} we define the convex sets $H_x$ \eqref{eq:lw.cvx},
prove a special separation property (Lemma~\ref{lm:lw}) and introduce our
family of methods \eqref{eq:plw}. Moreover, the first convergence analysis
results are obtained, namely: monotonicity (Proposition~\ref{pr:plw.mon})
and strong convergence (Theorems~\ref{th:lw.sc} and~\ref{th:lw.lim}) for
the exact data case.
In Section~\ref{sec:pli} we consider the noisy data case ($\delta > 0$). The
convex sets $H_x^\delta$ are defined and another separation property is derived
(Lemma~\ref{lm:pi.0}). The discrepancy principle is used to define a stopping criteria
\eqref{def:discrep}, which is proved to be finite (Theorem~\ref{th:plw-noise}).
Monotonicity is proven (Proposition~\ref{pr:plw.monE}) as well as a stability result
(Theorem~\ref{th:lw.stab}) and a norm convergence result (Theorem~\ref{th:lw.scE}).
Section~\ref{sec:numerics} is devoted to numerical experiments.
In Section~\ref{sec:conclusion} we present final remarks and conclusions.

\section{Main assumptions and preliminary results}
\label{sec:nbr}

In this section we state our main assumptions and discuss some of their
consequences, which are relevant for the forthcoming analysis. To simplify
the notation, from now on we write
\begin{equation} \label{eq:fdelta}
  F_\delta(x) \ := \ F(x) - y^\delta \quad\quad
  {\rm and }  \quad\quad F_0(x) \ := \ F(x) - y \, .
\end{equation}

Throughout this work we make the following assumptions, which are
frequently used in the analysis of iterative regularization methods
(see, e.g., \cite{EHN96, KNS08, Sch93a}):

\begin{description}
\item[A1] $F$ is a continuous operator defined on $D(F) \subset X$, which
has nonempty interior. Moreover, there exist constants $C$, $\rho > 0$ and
$x_0 \in D(F)$ such that $F'$, the Gateaux derivative of $F$, is defined
on $B_\rho(x_0)$ and satisfies
\begin{equation} \label{eq:a-dfb}
    \| F'(x) \| \ \le \ C \, , \quad \ x \in B_\rho(x_0) \subset D(F)
\end{equation} 
(the point $x_0$ is be used as initial guess for our family of methods).
\item[A2] The \emph{local tangential cone condition} (TCC) \cite{EHN96, KNS08}
\begin{equation} \label{eq:tcc}
  \| F(\bar{x}) - F(x) -  F'(x)( \bar{x} - x ) \|_Y \ \leq \
  \eta \norm{F(\bar{x}) - F(x)}_Y \, , \quad \forall \ x, \bar{x} \in B_{\rho}(x_0)
\end{equation}
holds for some $\eta < 1$, $x_0\in X$, and $\rho > 0$.
\item[A3] There exists an element $x^\star \in B_{\rho/2}(x_0)$ such that
$F(x^\star) = y$, where $y \in Rg(F)$ are the exact data satisfying
\eqref{eq:noisy-i}.
\item[A4] The operator $F$ is continuously Fr\'echet differentiable on $B_\rho(x_0)$.
\end{description}

Observe that in the TCC we require $\eta <1$, instead of $\eta<1/2$ as in
classical convergence analysis for the nonlinear Landweber under this
condition~\cite{EHN96}.
The TCC \eqref{eq:tcc} represents a uniform assumption (on a ball of radius $\rho$)
on the non-linearity of the operator $F$, and has interesting consequences
(see \cite[pg.278--280]{EHN96} or \cite[pg.6 and Sec.2.4 (pg.26--29)]{KNS08}).
Here we discuss some of them.

\begin{proposition}
  \label{pr:new}
  If \textbf{A1} and \textbf{A2} hold, then for any  $x$, $\bar{x} \in B_\rho(x_0)$

  \begin{enumerate}
  \item\label{it:1}
    $(1-\eta)\norm{F(x)-F(\bar{x})} \leq \norm{F'(x)(x-\bar{x})}\leq
    (1+\eta)\norm{F(x)-F(\bar{x})}$;
  \item\label{it:2}
    $ \inner{F'({x})^*F_0({x})}{{x}-\bar
      x}\leq(1+\eta)(\norm{F_0(x)}^2+\norm{F_0(x)}\norm{F_0(\bar{x})})$;
  \item\label{it:3}
    $\inner{F'({x})^*F_0({x})}{{x}-\bar{x}} \geq
    (1-\eta)\norm{F_0({x})}^2-(1+\eta)\norm{F_0({x})}\norm{F_0(\bar{x})}$.
  \end{enumerate}
  If, additionally, $F_0(x)\neq 0$ then
  \begin{multline*}
    (1-\eta) \norm{F_0(x)}-(1+\eta) \norm{F_0(\bar{x})} \\
    \leq \ \norm{F'(x)^*(x-\bar{x})} \
    \leq \ (1+\eta)(\norm{F_0(x)}+\norm{F_0(\bar{x})}) .
  \end{multline*}

\end{proposition}

\begin{proof}
  Item~\ref{it:1} follows immediately from the TCC and the triangle
  inequality, as proved in
  \cite[Eq.(11.7)]{EHN96}.

  Direct algebraic manipulations yield
  \begin{multline*}
    \inner{F'(x)^*F(x)}{x-\bar{x}} \ = \ \inner{F(x)}{F'(x)(x-\bar{x})} \\
    \leq \ (1+\eta) \norm{F(x)} \norm{F(x)-F(\bar{x})} \, ,
  \end{multline*}
  where the inequality follows from Cauchy-Schwarz inequality and
  item~\ref{it:1}. Likewise,
  \begin{align*}
    \langle F'({x})^* & F_0({x}) , \, {x}-\bar{x} \rangle
    \ = \    \inner{F_0({x})}{F'({x})({x}-\bar{x})} \\
    & = \    \inner{F_0({x})}{F_0({x})-F_0(\bar{x})}
           + \inner{F_0({x})}{F_0(\bar{x})-F_0(x)-F'(x)(\bar{x}-x)} \\
    & \geq \ \norm{F_0({x})}^2-\norm{F_0({x})}\norm{F_0(\bar{x})}
           - \eta \norm{F_0({x})}\norm{F_0({x})-F_0(\bar{x})} \, ,
  \end{align*}
  where the inequality follows from Cauchy-Schwarz inequality and the
  first inequality in this proof. Items \ref{it:2} and \ref{it:3}
  follow from the above inequalities and the inequality
  $\norm{F_0(x)-F_0(\bar{x}) }\leq\norm{F_0(x)}+\norm{F_0(\bar{x})}$.
\end{proof}

The next result relates to the solvability of operator equation
$F(x) = y$ with exact data.

\begin{proposition}
  \label{pr:fnz}
  Let \textbf{A1} -- \textbf{A3} be satisfied. For any $x \in B_\rho(x_0)$,
  $F_0(x) = 0$  if and only if  $F'(x)^* F_0(x) = 0$.
  Moreover, for any $(x_k) \in B_\rho(x_0)$ converging to some $\bar{x} \in B_\rho(x_0)$,
  the following statements are equivalent: \medskip

  \centerline{\hfil {a)} \ $\lim\limits_{k\to\infty} \norm{F'(x_k)^* F_0(x_k)} \, = \, 0$;
            \ \quad {b)} \ $\lim\limits_{k\to\infty} \norm{F_0(x_k)} \, = \, 0$;
            \ \quad {c)} $F(\bar{x}) \, = \, y$. \hfil}
\end{proposition}

\begin{proof}
  See \cite[pg.~279]{KNS08} for a proof of the first statement. 
  For proving the second statement: \\
  the implication $(b) \Rightarrow (a)$ follows from \textbf{A1}
  and the hypothesis $(x_k) \in B_\rho(x_0)$; on the other hand, 
  $(a) \Rightarrow (b)$ follows from Proposition~\ref{pr:new},
  item~\ref{it:3} with $x = x_k$ and $\bar{x} = x^\star$;
  moreover, $(b) \Rightarrow (c)$ and $(b) \Leftarrow (c)$ follow
  from the hypothesis \ $\lim\limits_{k\to\infty} \| x_k - \bar{x} \| = 0$
  and \textbf{A1}.
\end{proof}

Notice that the equivalence between $(a)$ and $(b)$ in Proposition~%
\ref{pr:fnz} does not depend on the convergence of sequence $(x_k)$. The next
result provides a convenient way of rewriting the TCC \eqref{eq:tcc}
for $\bar{x} = x^\star \in F^{-1}(y)$ using notation~\eqref{eq:fdelta}.

\begin{proposition}
  \label{pr:tcc.e}
  Let \textbf{A2} be satisfied. If $x^\star\in B_\rho(x_0)\cap F^{-1}(y)$ then
\begin{align*}
  \norm{y-y^\delta-F_\delta(x)-F'(x)(x^\star-x)} \ \leq \
  \eta \, \norm{y-y^\delta-F_\delta(x)},\quad\forall x\in B_\rho(x_0).
\end{align*}
\end{proposition}

\section{A family of relaxed projection Landweber methods}
\label{sec:ple}

In this section we assume that exact data $y^\delta = y \in$ Rg$(F)$ are
available, introduce a family of relaxed projection Landweber methods
for the exact data case, and prove their convergence. 

Define, for each $x\in D(F)$, the set
\begin{align}
  \label{eq:lw.cvx}
    H_x \ := \ \set{ z \in X \; | \;
        \inner{z-x}{F'(x)^*F_0(x)}\leq -(1-\eta)\norm{F_0(x)}^2} \, .  
\end{align}
Note that $H_x$ is either $\emptyset$, a closed half-space, or $X$. As we prove next, 
$H_x$ has an interesting geometric feature: it contains all exact solutions of
\eqref{eq:inv-probl} in $B_\rho(x_0)$ and, whenever $x$ is not a solution of
\eqref{eq:inv-probl}, it does not contain $x$.

\begin{lemma}[Separation]
  \label{lm:lw}
  Let \textbf{A1} and \textbf{A2} be satisfied. If $x \in B_\rho(x_0)$ then
  \begin{align} \label{eq:ldw.1}
    0 \ \geq \ (1-\eta) \norm{F_0(x)}^2 + \inner{F'(x)^*F_0(x)}{x^\star-x} \, ,
    \quad \forall x^\star\in B_\rho(x_0)\cap F^{-1}(y) \, .
  \end{align}
  Consequently, \medskip

  \centerline{1. \ $B_\rho(x_0)\cap F^{-1}(y) \subset H_x$; \hfil
  2. \ $x\in H_x\iff F(x) = y$.}
\end{lemma}

\begin{proof}
 The first statement follows trivially from Proposition~\ref{pr:new} item \ref{it:3}
 with $\bar{x} = x^\star \in B_\rho(x_0) \cap F^{-1}(y)$.
 Items 1 and 2 are immediate consequences of the first statement and Definition~%
 \eqref{eq:lw.cvx}.
\end{proof}

We are now ready to introduce our family of relaxed projection
Landweber methods. Choose $x_0 \in X$ according to \textbf{A2}
and \textbf{A3} and define, for $k \geq 0$, the sequence
\begin{subequations}
  \label{eq:plw}
\begin{align}
    \label{eq:plw.a}
  &x_{k+1} :=  x_k - \theta_k \,\lambda_k
    \, F'(x_k)^* F_0(x_k), \\
  \label{eq:plw.b}
  &\text{where }\theta_k \in (0, 2),\;\lambda_k:=
    \begin{cases}
      0, &\text{if }  F'(x_k)^* F_0(x_k) = 0 \\
      \dfrac{(1-\eta) \norm{F_0(x_k)}^2}{\norm{F'(x_k)^*F_0(x_k)}^2},
      & \text{otherwise.} 
    \end{cases}
\end{align}
\end{subequations}
In view of definition \eqref{eq:lw.cvx}, the orthogonal projection of
$x_k$ onto $H_{x_k}$ is $\hat{x} = x_k - \lambda_k F'(x_k)^*F_0(x_k)$
so that
\begin{align*}
  x_{k+1} = x_k + \theta_k(\hat{x} - x_k) = (1-\theta_k) x_k + \theta_k \hat{x} \, .
\end{align*}

We define the \emph{PLW method} as \eqref{eq:plw}  with $\theta_k = 1$ for all $k$.
This choice amounts to taking $x_{k+1}$ as the orthogonal projection of $x_k$ onto
$H_{x_k}$. The family of \emph{relaxed projection Landweber methods} is obtained by
choosing $\theta_k\in(0,2)$, which is equivalent to taking $x_{k+1}$ as a relaxed
orthogonal projection of $x_k$ onto $H_{x_k}$.

Iteration  \eqref{eq:plw} is well defined for all $x_k \in D(F)$; due to
Proposition~\ref{pr:fnz}, this iteration becomes stationary at
$x_{\tilde k} \in B_\rho(x_0)$, i.e. $x_k = x_{\tilde k}$ for $k \geq \tilde k$,
if and only if  $F(x_{\tilde k}) = y$.

In the next proposition an inequality is established, that guarantees the
monotonicity of the iteration error for the family of relaxed projection
Landweber methods in the case of exact data, i.e.,
{\color{black} $\norm{x^\star - x_{k+1}}  \leq  \norm{x^\star - x_k}$,}
whenever $\theta_k \in (0,2)$.

\begin{proposition}
  \label{pr:plw.mon}
  Let \textbf{A1} -- \textbf{A3} hold true.
  If $x_k\in B_\rho(x_0)$, $F'(x_k)^*F(x_k) \neq 0$, and $\theta_k$ and
  $x_{k+1}$ are as in \eqref{eq:plw}, then
  \begin{align*}
    \norm{x^\star-x_k}^2 \
    & \geq \ \norm{x^\star - x_{k+1}}^2  + \theta_k \, (2-\theta_k)
    \left( (1-\eta) \dfrac{\norm{F_0(x_k)}^2}{\norm{F'(x_k)^*F_0(x_k)}} \right)^2 ,
  \end{align*}
  for all $x^\star\in  B_\rho(x_0)\cap F^{-1}(y)$.
\end{proposition}

\begin{proof}
  If $x_k\in B_\rho(x_0)$ and $F'(x_k)^*F(x_k) \neq 0$ then $x_{k+1}$ is
  a relaxed orthogonal projection of $x_k$ onto the $H_{x_k}$ with a
  relaxation factor $\theta_k$.
  The conclusion follows from this fact,  Lemma~\ref{lm:lw}, iteration formula
  \eqref{eq:plw}
  { \color{black} and the properties of relaxed metric projections onto arbitrary
  sets (see, e.g., \cite[Lemma 3.13, pp. 21--22]{VaEr09}.) }
\end{proof}

Direct inspection of the inequality in Proposition~\ref{pr:plw.mon} shows that 
the choice $\theta_k\in(0,2)$, as prescribed in \eqref{eq:plw}, guarantees
decrease of the iteration error $\norm{x^\star - x_k}$, while $\theta_k = 1$
yields the greatest \emph{estimated} decrease on the iteration error.

We are now ready to state and prove the main results of this section:
Theorem~\ref{th:lw.sc} gives a sufficient condition for strong convergence
of the family of relaxed projection Landweber methods (for exact data)
to \emph{some} point $\bar{x} \in B_\rho(x_0)$. Theorem~\ref{th:lw.lim}
gives a sufficient condition for strong convergence of this family of
methods to \emph{a solution} of $F(x) = y$, shows that steepest descent,
minimal error, as well as Landweber method are particular instances of
methods belonging to this family, and proves convergence of these three
methods within this framework.

Recall that the steepest descent method (SD) is given by
$$
x_{x+1} \ = \ x_k - \dfrac{\norm{F'(x_k)^* F_0(x_k)}^2}
                          {\norm{F'(x_k) F'(x_k)^* F_0(x_k)}^2} \, F'(x_k)^* F_0(x_k) \, ,
$$
while the minimal error method (ME) is given by
$$
x_{x+1} \ = \ x_k - \dfrac{\norm{F_0(x_k)}^2}
                          {\norm{F'(x_k)^* F_0(x_k)}^2} \, F'(x_k)^* F_0(x_k) \, .
$$

\begin{theorem}
  \label{th:lw.sc}
  If \textbf{A1} -- \textbf{A3} hold true, then  the sequences
  $(x_k)$, $(\theta_k)$ as specified in \eqref{eq:plw} are well defined and
  \begin{align} \label{eq:lw.ball}
    x_k \in B_{\rho/2}(x^\star) \subset B_\rho(x_0) \, , \ \forall \ k \in \N \, .
  \end{align}
  If, additionally,  $\sup\,\theta_k < 2$, then 
  \begin{align} \label{eq:sf.t}
    (1-\eta)^2 \sum_{k=0}^\infty
    \theta_k \dfrac{\norm{F_0(x_k)}^4}{\norm{F'(x_k)^*F_0(x_k)}^2}  < \ \infty 
  \end{align}
  and $(x_k)$ converges strongly to some $\bar{x} \in B_\rho(x_0)$.
\end{theorem}

\begin{theorem}
  \label{th:lw.lim}
  Let \textbf{A1} -- \textbf{A3} hold true, and the sequences
  $(x_k)$, $(\theta_k)$ be defined as in \eqref{eq:plw}. The following
  statements hold: \medskip

  \noindent
  {\it a)} If \ $\inf$ $\theta_k > 0$ \ and \ $\sup$ $\theta_k < 2$,
  then $(x_k)$ converges to some $\bar{x} \in B_\rho(x_0)$ solving $F(\bar{x}) = y$.
  \medskip

  \noindent
  {\it b)} If \textbf{A2} holds with $\eta < 1/2$ and
  $$
  \theta_k :=(1-\eta)^{-1} \frac{\norm{F'(x_k)^*F_0(x_k)}^2}{\norm{F_0(x_k)}^2} \cdot
  \frac{\norm{F'(x_k)^*F_0(x_k)}^2}{\norm{F'(x_k) F'(x_k)^*F_0(x_k)}^2} ,
  $$
  then $0 < \theta_k \leq(1-\eta)^{-1} < 2$,
  iteration \eqref{eq:plw} reduces to the steepest descent method
  and $(x_k)$ converges to some $\bar{x} \in B_\rho(x_0)$ solving $F(\bar{x}) = y$.
  \medskip

  \noindent
  {\it c)} If \textbf{A1} and \textbf{A2} hold with $C \leq 1$ and
  $\eta < 1/2$, respectively, and
  $$
  \theta_k := (1-\eta)^{-1} \frac{\norm{F'(x_k)^*F_0(x_k)}^2}{\norm{F_0(x_k)}^2} ,
  $$
  then $0 < \theta_k \leq (1-\eta)^{-1} < 2$, iteration \eqref{eq:plw}
  reduces to the nonlinear Landweber iteration and $(x_k)$ converges to
  some $\bar{x} \in B_\rho(x_0)$ solving $F(\bar{x}) = y$.
  \medskip

  \noindent
  {\it d)} If \textbf{A1} and \textbf{A2} hold with $C \leq 1$ and
  $\eta < 1/2$, respectively, and $\theta_k := (1-\eta)^{-1}$, then
  iteration \eqref{eq:plw} reduces to the nonlinear minimal error method
  and $(x_k)$ converges to some $\bar{x} \in B_\rho(x_0)$ solving $F(\bar{x}) = y$.
\end{theorem}

\begin{proof}\mbox{\it (Theorem~\ref{th:lw.sc}) }
  Assumption \textbf{A3} guarantees the existence of $x^\star \in
  B_{\rho/2}(x_0)$, a solution of $F(x) = y$. It follows from \textbf{A3} that
  \eqref{eq:lw.ball} holds for $k=0$. Suppose that the sequence $(x_k)$, is
  well defined up to $k_0$ and that \eqref{eq:lw.ball} holds for $k = k_0$. It
  follows from \textbf{A1} that $x_{k_0} \in D(F)$, so that $x_{k_0+1}$ is well
  defined while it follows from \eqref{eq:plw} and
  Proposition~\ref{pr:plw.mon} that \eqref{eq:lw.ball} also holds for
  $k = k_0 + 1$.

  To prove the second part of the theorem,  suppose that $b := \sup\theta_k < 2$.
  At this point we have to consider two separate cases: \\
  {\bf Case I:} \ $F(x_{\tilde k}) = y$ for some $\tilde k \in \N$. \\
  It follows from \eqref{eq:lw.ball}, Proposition~\ref{pr:fnz} and \eqref{eq:plw},
  that $x_j = x_{\tilde k}$ for $j \geq \tilde k$, and we have trivially strong
  convergence of $(x_k)$ to $\bar{x} = x_{\tilde k}$ (which, in this case, is a
  solution of $F(x) = y$). \\
  {\bf Case II:} \ $F(x_k)\neq y$, for all $k$. \\
  It follows from \eqref{eq:lw.ball} and Proposition~\ref{pr:fnz} that
  $F'(x_k)^* F_0(x_k) \neq 0$ for all $k$. 
  According to~\eqref{eq:plw.b} 
  \begin{equation} \label{eq:xk-th-lb}
  \lambda_k \, := \, (1-\eta) \, \norm{F_0(x_k)}^2 \, \norm{F'(x_k)^*F_0(x_k)}^{-2}.
  \end{equation}
  Since $0 < \theta_k \leq b < 2$ for all $k$,
  $(2-\theta_k) \theta_k \geq (2-b) \theta_k > 0$, for all $k$. Therefore,
  it follows from Proposition~\ref{pr:plw.mon} that
  \begin{align*}
  \norm{x^\star-x_k}^2 + (2-b) \, \theta_k \, (1-\eta)^2  \sum_{j=0}^{k-1}
    \left(\dfrac{\norm{F_0(x_j)}^2}{\norm{F'(x_j)^*F_0(x_j)}}\right)^2
    \ \leq \ \norm{x^\star-x_0}^2 ,
  \end{align*}
  for all $x^\star\in B_\rho(x_0) \cap F^{-1}(y)$ and all $k\geq 1$. Consequently,
  using the definition of $\lambda_k$, we obtain
  \begin{align} \label{eq:sf}
    (1-\eta)^2 \sum_{k=0}^\infty
    \theta_k \dfrac{\norm{F_0(x_k)}^4}{\norm{F'(x_k)^*F_0(x_k)}^2} \ = \
    (1-\eta) \sum_{k=0}^\infty
    \theta_k\lambda_k\norm{F_0(x_k)}^2 \ < \ \infty \, , 
  \end{align}
  which, in particular, proves~\eqref{eq:sf.t}.
  
  If $\sum \theta_k\lambda_k<\infty$ then $\sum \norm{x_k - x_{k+1}} < \infty$
  (due to \eqref{eq:plw.a} and \textbf{A1}) and $(x_k)$ is a Cauchy sequence.

  Suppose that $\sum\theta_k\lambda_k = \infty$. It follows from \eqref{eq:sf}
  that $\lim\inf\norm{F_0(x_k)}=0$. Since we are in Case II, the sequence
  $(\norm{F_0(x_k)})$ is strictly positive and there exists a subsequence
  $(x_{\ell_i})$ satisfying
  \begin{equation} \label{eq:mprop-csarg}
  0 \leq k \leq \ell_i \ \Rightarrow \ \norm{F_0(x_k)} \ \geq \ \norm{F_0(x_{\ell_i})} \, .
  \end{equation}
  For all $k \in \N$ and $z \in B_\rho(x_0)$,
  \begin{align}
    \norm{x_k - z}^2
    & = \norm{x_{k+1} - z}^2 - \norm{x_k - x_{k+1}}^2
        - 2\inner{x_k - x_{k+1}}{x_k -z}  \nonumber \\
    & \leq \norm{x_{k+1} -z}^2 - 2\inner{x_k - x_{k+1}}{x_k - z} \nonumber \\
    & = \norm{x_{k+1} - z}^2
        + 2\theta_k \lambda_k \inner{F'(x_k)^*F_0(x_k)}{x_k - z} \nonumber \\
    & \leq \norm{x_{k+1} -z}^2 + 8 \lambda_k
      ( \norm{F_0(x_k)}^2 + \norm{F_0(x_k)} \norm{F_0(z)} ) \, , \label{eq:ufa!} 
  \end{align}
  where the second equality follows from \eqref{eq:plw.a} and the last
  inequality follows from Proposition~\ref{pr:new}, item~\ref{it:2},
  and the assumption $\eta < 1$.
  Thus, taking $z = x_{\ell_i}$ in \eqref{eq:ufa!}, we obtain
  \begin{align*}
    \norm{x_k - x_{\ell_i}}^2 \ \leq \ \norm{x_{k+1}-x_{\ell_i}}^2
    + 16 \lambda_k \norm{F_0(x_k)^2} \, , \ {\rm for} \ 0 \leq k < \ell_i \, .
  \end{align*}
  Define $s_m = \sum_{k\geq m} \theta_k \lambda_k \norm{F_0(x_k)}^2$. It follows from
  \eqref{eq:sf} that $\lim\limits_{m\to\infty} s_m = 0$.
  If $0 \leq k < \ell_i$, by adding the above inequality for
  $j = k$, $k+1$, $\ldots$, $\ell_i-1$, we get
  \begin{align*}
    \norm{x_k-x_{\ell_i}}^2 \ \leq \ 16 \sum_{j=k}^{\ell_i-1} \lambda_j \norm{F_0(x_j)}^2
    \ \leq \ 16 s_k \, .
  \end{align*}
  Now, take $k<j$. There exists $\ell_i>j$. Since $s_k>s_j$,
  $$
  \norm{x_k-x_j} \ \leq \ \norm{x_k-x_{\ell_i}} + \norm{x_j-x_{\ell_i}}
  \ \leq \ 4 \sqrt{s_k} + 4\sqrt{s_j} \ \leq \ 8 \sqrt{s_k} \, .
  $$
  Therefore, $(x_k)$ is a Cauchy sequence and converges to some element $\bar{x} \in
  \overline{B_\rho(x_0)}$.
\end{proof}

\begin{proof}\mbox{\it (Theorem~\ref{th:lw.lim}) }
  It follows from the assumptions of statement (a), from Theorem~\ref{th:lw.sc}, and
  from \textbf{A1} that $(x_k)$ converges to some $\bar{x} \in B_\rho(x_0)$ and that
  \[
  0=\lim_{k\to\infty}\dfrac{\norm{F_0(x_k)}^4}{\norm{F'(x_k)^*F_0(x_k)}^2} \geq
  \lim\sup_{k\to\infty}\dfrac{\norm{F_0(x_k)}^2}{C^2}.
  \]  
  Assertion (a) follows now from Proposition~\ref{pr:fnz}.

  To prove item (b), first use Cauchy-Schwarz inequality to obtain
  \begin{align*}
    0 < \dfrac{\norm{F'(x_k)^*F_0(x_k)}^4}{\norm{F_0(x_k)}^2{\norm{F'(x_k) F'(x_k)^*
          F_0(x_k)}^2}}
    \leq
    \dfrac{\norm{F'(x_k)^*F_0(x_k)}^4}
    {\inner{F_0(x_k)}{F'(x_k) F'(x_k)^*
      F_0(x_k)}^2}=1
  \end{align*}
  Therefore, $0<\theta_k\leq (1-\eta)^{-1}<2$ for all $k$ and it follows from
  Theorem~\ref{th:lw.sc}, the definition of $\theta_k$ and
  from \textbf{A1} that $(x_k)$ converges to some
  $\bar{x} \in B_\rho(x_0)$ and that
  \[
  0=\lim_{k\to\infty}\dfrac{\norm{F_0(x_k)}^2\norm{F'(x_k)^*F_0(x_k)}^2}
  {\norm{F'(x_k)F'(x_k)^*F_0(x_k)}^2} \geq
  \lim\sup_{k\to\infty}\dfrac{\norm{F_0(x_k)}^2}{C^2}.
  \]  
  Assertion (b)
  follows now from Proposition~\ref{pr:fnz}. \\
  It follows from the assumptions of statement (c) that $0 < \theta_k < (1-\eta)^{-1} < 2$.
  From this point on, the proof of statement (c) is analogous to the proof of
  statement (b). \\
  It follows from the assumptions of statement (d) that $0 < \theta_k < 2$. As before,
  the proof of statement (d) is analogous to the proof of statement (b).
\end{proof}

\begin{remark}
The argument used to establish strong convergence of sequence $(x_k)$ in the proof
of Theorem~\ref{th:lw.sc} is inspired in the technique used in
\cite[Theorem~2.3]{HNS95} to prove an analog result for the nonlinear Landweber
iteration.
Both proofs rely on a \emph{Cauchy sequence argument} (it is necessary to prove that
$(x_k)$ is a Cauchy sequence).
In \cite{HNS95}, given $j \geq k$ arbitrarily large, an element $j \geq l \geq k$
is chosen with a minimal property (namely, $\| F_0(x_l) \| \leq \| F_0(x_i) \|$,
for $k \leq i \leq j$). 
 In the proof of Theorem~\ref{th:lw.sc}, the auxiliary indexes $\ell_i$ defined in
 \eqref{eq:mprop-csarg} play a similar role. These indexes are also chosen according
 to a minimizing property, namely,  the subsequence $( \| F_0(x_{\ell_j}) \| )$
 is monotone non-increasing.
\end{remark}

\section{Convergence analysis: noisy data}
\label{sec:pli}

In this section we analyse the family of relaxed projected Landweber methods
in the noisy data case and investigate convergence properties.
We assume that only noisy data $y^\delta \in Y$ satisfying
\eqref{eq:noisy-i} are available, where the noise level $\delta > 0$ is known.
Recall that to simplify the presentation we are using notation
\eqref{eq:fdelta}, i.e., $F_\delta(x) = F(x) - y^\delta$.

Since $F_0(\cdot) = F(\cdot) - y$ is not available, one can not compute the
projection onto $H_x$ (defined in Section~\ref{sec:ple}).
Define, instead, for each $x \in B_\rho(x_0)$, the set
\begin{multline}
  \label{eq:lw.cvx.d}
  H^\delta_x \ := \ \Big\{ z \in X \; \Big| \; \inner{z-x}{F'(x)^*F_\delta(x)} \ \leq
  \\ \leq \
  - \norm{F_\delta(x)} \Big( (1-\eta) \Norm{F_\delta(x)} - (1+\eta) \delta \Big) \Big\} \, .  
\end{multline}
Next we prove a ``noisy'' version of the separation Lemma~\ref{lm:lw}:
$H^\delta_x$ contains all exact solutions of $F(x)=y$ (within $B_\rho(x_0)$) and,
if the residual $\norm{F_\delta(x)}$ is above the threshold
$(1+\eta)(1-\eta)^{-1}\delta$, then $H^\delta_x$ does not contain $x$.

\begin{lemma}[Separation]
  \label{lm:pi.0}
  Suppose that \textbf{A1} and \textbf{A2} hold. If $x \in B_\rho(x_0)$, then
  \begin{equation} \label{eq:ldw.S}
  0 \geq \norm{F_\delta(x)} \big[ (1 - \eta)\, \norm{F_\delta(x)}
         - (1+\eta)\, \delta \big]
    + \inner{x^\star-x}{F'(x)^*F_\delta(x)} \, ,
  \end{equation}
  for all $x^\star\in B_\rho(x_0)\cap F^{-1}(y)$.
  Consequently, $B_\rho(x_0) \cap F^{-1}(y) \subset H^\delta_x$.
\end{lemma}

\begin{proof}
  Indeed, for $x^\star\in B_\rho(x_0)\cap F^{-1}(y)$ we have
  \begin{align*}
    \langle F'(x)^* & F_\delta(x) , \, x^\star-x \rangle
      = \    \inner{F_\delta(x)}{F'(x)(x^\star-x)} \\
    & = \    \inner{F_\delta(x)}{F_\delta(x)+F'(x)(x^\star-x)}-\norm{F_\delta(x)}^2 \\
    & = \    \inner{F_\delta(x)}{F_0(x)+F'(x)(x^\star-x)} +\inner{F_\delta(x)}{y-y^\delta}
             \norm{F_\delta(x)}^2 \\
    & \leq \ \norm{F_\delta(x)} \ \eta\, \norm{F_0(x)} + \norm{F_\delta(x)}\, \delta
           - \norm{F_\delta(x)}^2
  \end{align*}
  where the first inequality follows from Cauchy-Schwarz inequality and \eqref{eq:tcc}.
  Since $\norm{F_0(x)}\leq \norm{F_\delta(x)}+\delta$,
  \[
  \inner{F'(x)^*F_\delta(x)}{x^\star-x} \leq 
  \eta\norm{F_\delta(x)} (\norm{F_\delta(x)}+\delta) + \norm{F_\delta(x)}\, \delta
  - \norm{F_\delta(x)}^2
  \]
  which is equivalent to  \eqref{eq:ldw.S}.
\end{proof}

Since $\norm{F_\delta(x)}>(1+\eta)(1-\eta)^{-1}\delta$ is sufficient for
separation of $x$ from $F^{-1}(y)$ in $B_\rho(x_0)$ via $H_x^\delta$, this
condition also guarantees $F'(x)^*F_\delta(x)\neq 0$. 

The iteration formula for the family of relaxed projection Landweber methods
in the noisy data case is given by
\begin{equation}
  \label{eq:plwE}
  x_{k+1}^\delta \ := \ x_k^\delta -
  \theta_k \dfrac{ p_{\delta}(\norm{F_\delta(x_k^\delta)})}
           {\norm{F'(x_k^\delta)^* F_\delta(x_k^\delta)}^2}
           F'(x_k^\delta)^* F_\delta(x_k^\delta) \, , \ \theta_k \in (0,2) \, ,
\end{equation}
where
\begin{equation} \label{def:tau-a-b.p}
  p_{\delta}(t) \ := \ t \, ( (1-\eta)t - (1+\eta)\delta )
\end{equation}
and the initial guess $x_0^\delta \in X$ is chosen according to \textbf{A1}.
Again, the \emph{PLW method} (for inexact data) is obtained by taking
$\theta_k = 1$, which amounts to define $x_{k+1}^\delta$ as the orthogonal
projection of $x_k^\delta$ onto $H_{x_k^\delta}^\delta$. On the other
hand, the relaxed variants, which use $\theta_k \in (0,2)$, correspond
to setting $x_{k+1}^\delta$ as a relaxed projection of $x_k^\delta$ onto
$H_{x_k^\delta}^\delta$.

Let
\begin{equation} \label{def:tau-a-b.t}
  \tau        \ > \ \dfrac{1+\eta}{1-\eta} \, .
\end{equation}
The computation of the sequence $(x_k^\delta)$ should be stopped at the index
$k_*^\delta \in \N$ defined by the discrepancy principle
\begin{equation}
  \label{def:discrep}
  k_*^\delta \ := \ \max \, \big\{ k \in \N \, ; \ \| F_\delta(x_{j}^\delta) \|
                  > \tau \delta \, , \ j = 0,1,\dots,k-1 \big\} \, .
\end{equation}
Notice that if $\| F_\delta(x_{k}^\delta) \| > \tau \delta$, then
$\| F'(x_k^\delta)^* F_\delta(x_k^\delta) \| \not = 0$. 
This fact is a consequence of Proposition~\ref{pr:new},
item~\ref{it:3}, since $F_\delta$ also satisfies \textbf{A1} and
\textbf{A2}.
Consequently, iteration \eqref{eq:plwE} is well defined for $k = 0,
\dots, k_*^\delta$.

The next two results have interesting consequences. From Proposition~%
\ref{pr:plw.monE} we conclude that ($x_k^\delta)$ does not leave the ball
$B_\rho(x_0)$ for $k = 0, \dots, k_*^\delta$. On the other hand, it follows
from Theorem~\ref{th:plw-noise} that the stopping index
$k_*^\delta$ is finite, whenever $\delta > 0$.

\begin{proposition}
  \label{pr:plw.monE}
  Let \textbf{A1} -- \textbf{A3} hold true and $\theta_k$ be chosen as in
  \eqref{eq:plwE}. If $x_k^\delta \in B_\rho(x_0)$ and $\| F_\delta(x_{k}^\delta) \|
  > \tau \delta$, then
  \begin{align*}
    \norm{x^\star - x_k^\delta}^2 \ \geq \ \norm{x^\star - x_{k+1}^\delta}^2
    + \theta_k(2-\theta_k) \, \left( \dfrac{p_\delta(\norm{F_\delta(x_k^\delta)})}
    {\norm{F'(x_k^\delta)^* F_\delta(x_k^\delta)}} \right)^2 ,
  \end{align*}
  for all $x^\star \in B_\rho(x_0) \cap F^{-1}(y)$.
\end{proposition}

\begin{proof}
  If $x_k^\delta \in B_\rho(x_0)$ and $\| F_\delta(x_{k}^\delta) \| > \tau \delta$,
  then $x_{k+1}^\delta$ is a relaxed orthogonal projection of $x_k^\delta$ onto
  $H^\delta_{x_k^\delta}$ with a relaxation factor $\theta_k$. The conclusion follows
  from this fact, Lemma~\ref{lm:pi.0}, the iteration formula~\eqref{eq:plwE}, and elementary
  properties of over/under relaxed orthogonal projections.
\end{proof}

\begin{theorem}
  \label{th:plw-noise}
  If \textbf{A1} -- \textbf{A3} hold true, then the sequences $(x_k^\delta)$,
  $(\theta_k)$ as specified in \eqref{eq:plwE} (together with the stopping
  criterion~\eqref{def:discrep}) are well defined and
  \begin{align*}
    x_k\in B_{\rho/2}(x^\star)\subset B_\rho(x_0),\;\;\;\forall k\leq k^\delta_*.
  \end{align*}
  Moreover, if $\theta_k\in[a,b]\subset(0,2)$ for all $k\leq k^\delta_*$, then
  this stopping index $k_*^\delta$ defined in \eqref{def:discrep} is finite.
\end{theorem}

\begin{proof}
  The proof of the first statement is similar to the one in Theorem~\ref{th:lw.sc}.

  To prove the second statement, first observe that since $\theta_k\in[a,b]$,
  $\theta_k(2-\theta_k)\geq a(2-b)>0$. Thus,  it follows from Proposition
  \ref{pr:plw.monE} that for any $k< k^\delta_*$
  \begin{multline*}
    \|x^\star - x_0^\delta \|^2 \ \geq \ 
    a(2-b) \sum_{j=0}^k  \left( \dfrac{p_\delta(\norm{F_\delta(x_k^\delta)})}
    {\norm{F'(x_k^\delta)^* F_\delta(x_k^\delta)}} \right)^2 \\
 \geq \
 \dfrac{a(2-b)}{C^2} \sum_{j=0}^k  \left( \dfrac{p_\delta(\norm{F_\delta(x_k^\delta)})}
    {\norm{F_\delta(x_k^\delta)}} \right)^2 .
  \end{multline*}
  Observe that, if $t>\tau\delta$, then
  \begin{align*}
    \dfrac{p_\delta(t)}{t}=(1-\eta)t-(1+\eta)\delta>
    \left[\tau-\dfrac{1+\eta}{1-\eta}\right](1-\eta)
    \delta=:h>0.
  \end{align*}
  Therefore, for any $k < k^\delta_*$
  \begin{align*}
     \|x^\star - x_0^\delta \|^2 
   \geq 
    \dfrac{a(2-b)}{C^2}(k+1)h^2,
  \end{align*}
  so that $k^\delta_*$ is finite.
  \end{proof}

It is worth noticing that the Landweber method for noisy data
\cite[Chap.11]{EHN96}
(which requires $\eta < 1/2$, $C \leq 1$ in \textbf{A1} -- \textbf{A2})
using the discrepancy principle \eqref{def:discrep} with 
$$
\tau > 2\,\dfrac{1+\eta}{1-2\eta} > \dfrac{1+\eta}{1-\eta} \, ,
$$
corresponds to the PLW method, analyzed in Theorem~\ref{th:plw-noise}, with
$$
0 \ < \ \dfrac{p_\delta(\tau\delta)}{\rho^2}
\ \leq \
\theta_k \ = \
\dfrac{\norm{F'(x_k^\delta)^* F_\delta(x_k^\delta)}^2}
      {p_\delta(\norm{F_\delta(x_k^\delta)})}
\ \leq \
\dfrac{\tau}{(1-\eta)\tau - (1+\eta)} < 2
$$
(here the second inequality follows from \textbf{A3} and the third inequality
follows from Lemma~\ref{lm:pi.0}). Consequently, in the noisy data case, the
convergence analysis for the PLW method encompasses the Landweber iteration
(under the TCC condition) as a particular case.

In the next theorem we discuss a stability result, which is an essential
tool to prove the last result of this section, namely Theorem~\ref{th:lw.scE}
(semi-convergence of the PLW method). Notice that this is the first time
were the strong Assumption \textbf{A4} is needed in the text.

\begin{theorem}
  \label{th:lw.stab}
  Let \textbf{A1} -- \textbf{A4} hold true. For each fixed $k \in \N$,
  the element $x_k^\delta$, computed after kth-iterations of any method
  within the family of methods in \eqref{eq:plwE}, depends continuously
  on the data $y^\delta$.
\end{theorem}
\begin{proof}
From \eqref{def:tau-a-b.t}, \textbf{A1}, \textbf{A4} and Theorem~\ref{th:plw-noise},
it follows that the mapping $\varphi: D(\varphi) \to X$ with
\begin{align*}
& D(\varphi) := \{ (x,y^\delta,\delta) | x \in D(F) ;
                   \, \delta > 0 ;
                   \, \norm{y^\delta - y} \leq \delta ;
                   \, F'(x)^*(F(x) - y^\delta) \neq 0 \} , \\
& \varphi(x,y^\delta,\delta) := x - \dfrac{p_\delta(\norm{F(x)-y^\delta})}
                                    {\norm{F'(x)^* (F(x)-y^\delta)}^2}
                                    \, F'(x)^* (F(x)-y^\delta)
\end{align*}
is continuous on its domain of definition. Therefore, whenever the iterate
$x_k^\delta = \big( \varphi(\cdot,y^\delta,\delta) \big)^k(x_0)$ is well defined,
it depends continuously on $(y^\delta, \delta$).
\end{proof}

Theorem~\ref{th:lw.stab} together with Theorems~\ref{th:lw.sc} and~\ref{th:lw.lim}
are the key ingredients in the proof of Theorem~\ref{th:lw.scE}, which guarantees
that the stopping rule \eqref{def:discrep} renders the PLW iteration a regularization
method. The proof of Theorem~\ref{th:lw.scE} uses classical techniques from the
analysis of Landweber-type iterative regularization techniques (see, e.g.,
\cite[Theorem~11.5]{EHN96} or \cite[Theorem~2.6]{KNS08}) and thus is omitted.

\begin{theorem}
  \label{th:lw.scE}
  Let \textbf{A1} -- \textbf{A4} hold true, $\delta_j \to 0$ as $j\to\infty$, and
  $y_j := y^{\delta_j} \in Y$ be given with $\norm{y_j - y} \leq \delta_j$. If the PLW
  iteration \eqref{eq:plwE} is stopped with $k_*^j := k_*^{\delta_j}$ according to the
  discrepancy \eqref{def:discrep}, then $(x_{k_*^j}^\delta)$ converges strongly to a
  solution $\bar{x} \in B_\rho(x_0)$ of $F(x)=y$ as $j\to\infty$.
\end{theorem}

It is immediate to verify that the result in Theorem~\ref{th:lw.scE} extend to any
method within the family of relaxed projection Landweber methods \eqref{eq:plwE}.

\section{Numerical experiments} \label{sec:numerics}

In what follows we present numerical experiments for the iterative methods
derived in previous sections.
The PLW method is implemented for solving an exponentially ill-posed inverse problem
related to the Dirichlet to Neumann operator and its performance is compared against
the benchmark methods LW and SD.

\subsection{Description of the mathematical model} \label{ssec:num-description}

We briefly introduce a model which plays a key rule in inverse doping problems
with current flow measurements, namely the 2D {\em linearized stationary bipolar
model close to equilibrium}.

This mathematical model is derived from the drift diffusion equations by
linearizing the Voltage-Current (VC) map at $U \equiv 0$ \cite{Le06,BELM04},
where the function $U = U(x)$ denotes the applied potential to the semiconductor
device.%
\footnote{This simplification is motivated by the fact that, due to hysteresis
effects for large applied voltage, the VC-map can only be defined as a single-valued
function in a neighborhood of $U=0$.}
Additionally, we assume that the electron mobility $\mu_n(x) = \mu_n > 0$ as
well as the hole mobility $\mu_p(x) = \mu_p > 0$ are constant and that no
recombination-generation rate is present \cite{LMZ06,LMZ06a}.
Under the above assumptions the Gateaux derivative of the VC-map $\Sigma_C$ at
the point $U=0$ in the direction $h \in H^{3/2}(\partial\Omega_D)$ is given by
\begin{equation} \label{eq:def-sigma-prime-C}
\Sigma'_C(0) h \ = \ \mu_n \, e^{V_{\rm bi}} \hat{u}_\nu
                   - \mu_p \, e^{-V_{\rm bi}} \hat{v}_\nu
               \ \in \ H^{1/2}(\Gamma_1) \, ,
\end{equation}
where the concentrations of electrons and holes $(\hat{u}, \hat{v})$ solve%
\footnote{These concentrations are here written in terms of the Slotboom variables
\cite{LMZ06a}.}
\begin{subequations}  \label{eq:bipol-stat}  \begin{eqnarray}
{\rm div}\, (\mu_n e^{V^0} \nabla \hat{u})   & \hskip-1.7cm
 = \ 0                \label{eq:bipol-statA} & {\rm in}\ \Omega \\
{\rm div}\, (\mu_p e^{-V^0} \nabla \hat{v})  & \hskip-1.7cm
 = \ 0                \label{eq:bipol-statB} & {\rm in}\ \Omega \\
\hat{u} & \hskip-0.1cm
 = \ -\hat{v} \ = \ -h                       & {\rm on}\ \partial\Omega_D \\
\nabla\hat{u} \cdot\nu &
 = \ \nabla\hat{v} \cdot\nu \ = \ 0          & {\rm on}\ \partial\Omega_N
\end{eqnarray} \end{subequations}
and the potential $V^0$ is the solution of the thermal equilibrium problem
\begin{subequations}  \label{eq:equil-case} \begin{eqnarray}
\lambda^2 \, \Delta V^0 &                       \label{eq:equil-caseA}
   = \ e^{V^0} - e^{-V^0} - C(x)              & {\rm in}\ \Omega \\
V^0 & \hskip-2.1cm                              \label{eq:equil-caseB}
   = \ V_{\rm bi}(x)                          & {\rm on}\ \partial\Omega_D \\
\nabla V^0 \cdot \nu & \hskip-2.95cm            \label{eq:equil-caseC}
   = \ 0                                      & {\rm on}\ \partial\Omega_N \, .
\end{eqnarray} \end{subequations}
Here $\Omega \subset \R^2$ is a domain representing the semiconductor
device; the boundary of $\Omega$ is divided into two nonempty disjoint parts:
$\partial\Omega = \overline{\partial\Omega_N} \cup \overline{\partial\Omega_D}$.
The Dirichlet boundary part $\partial\Omega_D$ models the Ohmic contacts,
where the potential $V$ as well as the concentrations $\hat{u}$ and $\hat{v}$
are prescribed; the Neumann boundary part $\partial\Omega_N$ corresponds to
insulating surfaces, thus zero current flow and zero electric field in
the normal direction are prescribed; the Dirichlet boundary part splits
into $\partial\Omega_D = \Gamma_0 \cup \Gamma_1$, where the disjoint curves
$\Gamma_i$, $i = 0$, $1$, correspond to distinct device contacts (differences
in $U(x)$ between segments $\Gamma_0$ and $\Gamma_1$ correspond to the applied
bias between these two contacts). Moreover, $V_{\rm bi}$ is a given logarithmic
function \cite{BELM04}.

The piecewise constant function $C(x)$ is the {\em doping profile} and models
a preconcentration of ions in the crystal, so $C(x) = C_{+}(x) - C_{-}(x)$ holds,
where $C_{+}$ and $C_{-}$ are (constant) concentrations of negative and positive
ions respectively.

In those subregions of $\Omega$ in which the preconcentration of negative
ions predominate (P-regions), we have $C(x) < 0$. Analogously, we define the
N-regions, where $C(x) > 0$ holds.
The boundaries between the P-regions and N-regions (where $C$ changes sign)
are called {\em pn-junctions}; it's determination is a strategic
non-destructive test \cite{LMZ06a, LMZ06}.

\subsection{The inverse doping problem}

The inverse problem we are concerned with consists in determining the
doping profile function $C$ in \eqref{eq:equil-case} from measurements
of the linearized VC-map $\Sigma'_C(0)$ in \eqref{eq:def-sigma-prime-C},
under the assumption $\mu_p = 0$ (the so-called {\em linearized stationary
unipolar model close to equilibrium}). Notice that we can split the inverse
problem in two parts:

\smallskip
{\bf 1)} Define the function $a(x) := e^{V^0(x)}$, $x \in \Omega$, and solve
the parameter identification problem
\begin{equation} \label{eq:num-d2nB}
{\rm div}\, (\mu_n a(x) \nabla \hat u) \ = \ 0 \,\ {\rm in} \ \Omega \quad\quad
\hat u \ = \ - U(x) \,\ {\rm on} \ \partial\Omega_D \quad\quad
\nabla \hat u \cdot \nu \ = \ 0 \,\ {\rm on} \ \partial\Omega_N .
%
%
\end{equation}
for $a(x)$, from measurements of \ 
$\big[ \Sigma'_C(0) \big](U) \, = \, \big( \mu_n a(x) \hat{u}_\nu \big)|_{\Gamma_1}$.
\smallskip
%

{\bf 2)} Evaluate the doping profile \
$C(x) = a(x) - a^{-1}(x) - \lambda^2 \Delta (\ln a(x))$, $x \in \Omega$.

\smallskip \noindent
Since the evaluation of $C$ from $a(x)$ can be explicitly performed in a stable way,
we shall focus on the problem of identifying the function parameter $a(x)$ in
(\ref{eq:num-d2nB}).
Summarizing, the inverse doping profile problem in the linearized stationary unipolar
model (close to equilibrium) reduces to the identification of the parameter function
$a(x)$ in (\ref{eq:num-d2nB}) from measurements of the Dirichlet-to-Neumann map \
$\Lambda_a : H^{1/2}(\partial\Omega_D) \ni U \mapsto
\big( \mu_n a(x) \hat{u}_\nu \big) |_{\Gamma_1} \in H^{-1/2}(\Gamma_1)$.
%

\smallskip In the formulation of the inverse problem we shall take into account
some constraints imposed by the practical experiments, namely:
(i) The voltage profile $U \in H^{1/2}(\partial\Omega_D)$ must satisfy
$U |_{\Gamma_1} = 0$ (in practice, $U$ is chosen to be piecewise constant
on $\Gamma_1$ and to vanish on $\Gamma_0$);
(ii) The identification of $a(x)$ has to be performed from a finite number of
measurements, i.e. from the data $\big\{ (U_i, \Lambda_a(U_i)) \big\}_{i=1}^N
\in \big[ H^{1/2}(\Gamma_0) \times H^{-1/2}(\Gamma_1) \big]^N$.

\smallskip In what follows we take $N = 1$, i.e. identification of $a(x)$
from a single experiment. Thus, we can write this particular inverse doping
problem within the abstract framework of \eqref{eq:inv-probl}
\begin{equation} \label{eq:ip-abstract}
 F(a) \ = \ \Lambda_a(U) \ =: \ y \, ,
\end{equation}
where $U$ is a fixed voltage profile satisfying the above assumptions, \ 
$X := L^2(\Omega) \supset D(F) := \{ a \in L^\infty(\Omega)$;
$0 < a_m \le a(x) \le a_M$, a.e. in $\Omega \}$ \ and \ 
$Y := H^{1/2}(\Gamma_1)$. The operator $F$ above is known to be continuous
\cite{BELM04}.

\subsection{First experiment: The Calderon setup}

In this subsection we consider the special setup $\Gamma_1 = \partial\Omega_D =
\partial\Omega$ (i.e., $\Gamma_0 = \partial\Omega_N = \emptyset$). Up to now, it
is not known whether the map $F$ satisfies the TCC.  However,
\begin{enumerate}
\item the map $a \mapsto u$ (solution of \eqref{eq:num-d2nB}) satisfies the TCC 
with respect to the $H^1(\Omega)$ norm~\cite{KNS08};
\item it was proven in \cite{LR08} that
the discretization of the operator $F$ in \eqref{eq:ip-abstract} using the finite element
method (and basis functions constructed by a Delaunay triangulation) satisfies the TCC
\eqref{eq:tcc}.  
\end{enumerate}
Therefore, the analytical convergence results of the previous sections do apply
to finite-element discretizations of \eqref{eq:ip-abstract} in this special setup.
Moreover, item 1 suggests that $H^1(\Omega)$ is a good choice of parameter space for
TCC based reconstruction methods.
%
Motivated by this fact,  the setup of the numerical experiments presented in this
subsection is chosen as follows:

\smallskip\noindent $\bullet$ The domain $\Omega \subset \mathbb R^2$ is the unit square
$(0,1) \times (0,1)$
and the above mentioned boundary parts are $\Gamma_1 = \partial\Omega_D := \partial\Omega$,
$\Gamma_0 = \partial\Omega_N := \emptyset$.

\smallskip\noindent $\bullet$ The parameter space is $H^1(\Omega)$ and the function
$a^\star(x)(x)$ to be identified is shown in
Figure~\ref{fig:calderon-setup}.

\smallskip\noindent $\bullet$ The fixed Dirichlet input for the DtN map \eqref{eq:num-d2nB}
is the continuous function $U: \partial\Omega \to \R$ defined by
$$
U(x,0) \ = \ U(x,1) \ := \  \sin(\pi x) \, , \quad
U(0,y) \ = \ U(1,y) \ := \ -\sin(\pi y)
$$
(in Figure~\ref{fig:calderon-setup}, $U(x)$ and the corresponding solution
$\hat u$ of (\ref{eq:num-d2nB}) are plotted).

\smallskip\noindent $\bullet$ The TCC constant $\eta$ in \eqref{eq:tcc} is not known for
this particular setup. In our computations we used the value $\eta = 0.45$ which
is in agreement with \textbf{A2}. \\ (Note that the convergence
analysis of the PLW method requires $\eta<1$ while the nonlinear LW method requires
the TCC with $\eta < 0.5$~\cite[Assumption (2.4)]{KNS08}.
The above choice allows the comparison of both methods.)

\smallskip\noindent $\bullet$ The ``exact data'' $y$ in \eqref{eq:ip-abstract} is
obtained by solving the direct problem \eqref{eq:num-d2nB} using a finite
element type method and adaptive mesh refinement (approx 131.000 elements).
In order to avoid inverse crimes, a coarser grid (with approx 33.000 elements)
was used in the finite element method implementation of the iterative methods.

\smallskip\noindent $\bullet$ In the numerical experiment with noisy data, artificially
generated (random) noise of 2\% was added to the exact data $y$ in order to generate
the noisy data $y^\delta$. For the verification of the stopping rule \eqref{def:discrep}
we assumed exact knowledge of the noise level and chose $\tau = 3$ in \eqref{def:tau-a-b.t},
which is in agreement with the above choice for $\eta$.

\begin{remark}[Choosing the initial guess] \label{rem:init_guess}
The initial guess $a_0(x)$ used for all iterative methods is presented
in Figure~\ref{fig:calderon-setup}. According to \textbf{A1} -- \textbf{A3},
$a_0(x)$ has to be sufficiently close to $a^\star(x)$ (otherwise the PLW
method may not converge). With this in mind, we choose $a_0(x)$ as the solution
of the Dirichlet boundary value problem
$$
\Delta a_0    \ = \ 0 \, ,    \ \mbox{in } \Omega \, , \quad\quad
       a_0(x) \ = \ U(x) \, , \ \mbox{at } \partial\Omega \, .
$$
This choice is an educated guess that incorporate the available {\em a priori}
knowledge about the exact solution $a^\star(x)$, namely: $a_0 \in H^1(\Omega)$
and $a_0(x) = a^\star(x)$ at $\partial\Omega_D$. Moreover, $a_0 = \argmin
\{ \norm{\nabla a}_{L^2(\Omega)}^2$ $|$ $a \in H^1(\Omega)$,
$ a^{}_{\partial\Omega}(x) = a^\star_{\partial\Omega}(x) \}$.
\end{remark}

\begin{remark}[Computing the iterative step] \label{rem:iter-step}
The computation of the $k$th-step of the PLW method (see \eqref{eq:plw}) requires the
evaluation of $F'(a_k)^* F_0(a_k)$. According to \cite{BELM04}, for all test
functions $v \in H^1_0(\Omega)$ it holds
$$
\inner{F'(a_k)^\times F_0(a_k)}{v}_{L^2(\Omega)}
\ = \ \inner{F_0(a_k)}{F'(a_k) v}_{L^2(\partial\Omega)}
\ = \ \inner{F_0(a_k)}{V}_{L^2(\partial\Omega)} \, ,
$$
where $F'(a_k)^\times$ stands for the adjoint of $F'(a_k)$ in $L^2(\Omega)$,
and $V \in H^1(\Omega)$ solves
$$
-\nabla \cdot (a_k(x) \, \nabla V) \ = \ \nabla \cdot (v \, \nabla F(a_k))
                                         \, , \ \mbox{in } \Omega \, , \quad\quad
                                   V \ = \ 0  \, , \ \mbox{at } \partial\Omega \, .
$$
Furthermore, in \cite{BELM04} it is shown that for all $\psi \in L^2(\partial\Omega)$
and $v \in H^1_0(\Omega)$
\begin{equation} \label{eq:F-adj}
\inner{F'(a_k)^\times \psi}{v}_{L^2(\Omega)}
\ = \ \inner{\psi}{V}_{L^2(\partial\Omega)}
\ = \ \inner{\nabla\Psi \cdot \nabla u_k}{v}_{L^2(\Omega)} \, ,
\end{equation}
where $\Psi$, $u_k$ $\in H^1(\Omega)$ solve
\begin{subequations} \label{eq:Psi_uk}
\begin{align}
-\nabla \cdot (a_k(x) \, \nabla \Psi) \ = \ 0 \, ,    \ \mbox{in } \Omega \, ,& \quad\quad
                                 \Psi \ = \ \psi \, , \ \mbox{at } \partial\Omega
                                 \label{eq:Psi} \\
-\nabla \cdot (a_k(x) \, \nabla u_k) \ = \ 0 \, ,    \ \mbox{in } \Omega \, ,& \quad\quad
                                 u_k \ = \ U(x) \, , \ \mbox{at } \partial\Omega \, .
\end{align}
\end{subequations}
respectively. An direct consequence of \eqref{eq:F-adj}, \eqref{eq:Psi_uk} is
the variational identity
$$
\inner{F'(a_k)^\times F_0(a_k)}{v}_{L^2(\Omega)}
\ = \ \inner{\nabla\Psi \cdot \nabla u_k}{v}_{L^2(\partial\Omega)} \, ,
\ \forall v \in H^1_0(\Omega) \, ,
$$
where $\Psi$ solves \eqref{eq:Psi} with $\psi = F_0(a_k)$.

Notice that $\nabla\Psi \cdot \nabla u_k$ is the adjoint, \emph{in $L^2(\Omega)$},
of $F'(a_k)$ applied to $F_0(a_k)$. 
We need to apply to $F_0(a_k)$, instead, the adjoint of $F'(a_k)$
\emph{in $H^1(\Omega)$}.
That is, we need to compute
$$
F'(a_k)^* F_0(a_k) = W_k \in H^1_0(\Omega) \, ,
$$
where $W_k$ is the Riesz vector satisfying \
$\inner{W_k}{v}_{H^1(\Omega)} = \inner{\nabla\Psi \cdot \nabla u_k}{v}_{L^2(\Omega)}$,
for all $v \in H^1(\Omega)$. A direct calculation yields
$$
(I-\Delta) \, W_k \ = \ \nabla\Psi \cdot \nabla u_k \, , \ \mbox{in } \Omega \, , \quad\quad
              W_k \ = \ 0 \, , \ \mbox{at } \partial\Omega \, .
$$
Within this setting, the PLW iteration \eqref{eq:plw} becomes
$$
a_{k+1} \ := \ a_k \, - \, (1-\eta)
             \dfrac{\norm{F_0(a_k)}_{L^2(\Omega)}^2}{\norm{W_k}_{H^1(\Omega)}^2} \ W_k \, .
$$

The iterative steps of the benchmark iterations LW and SD (implemented here for
the sake of comparison) are computed also using the adjoint of $F'(\cdot)$ in $H^1$.
Notice that a similar
argumentation can be derived in the noisy data case (see \eqref{eq:plwE}).
\end{remark}

For solving the elliptic PDE's above described, needed for the implementation of the
iterative methods, we used the package PLTMG \cite{Ban94} compiled with GFORTRAN-4.8
in a INTEL(R) Xeon(R) CPU E5-1650 v3.

\noindent {\bf First example:} Problem with exact data. \\
Evolution of both iteration error and residual is shown in Figure~\ref{fig:calderon-exact}.
The PLW method (GREEN) is compared with the LW method (BLUE) and with the
SD method (RED).
For comparison purposes, if one decides to stop iterating when
$\norm{F_0(a_k)} < 0.025$ is satisfied, the PLW method needs only 43
iterations, while the SD method requires 167 iterative steps and the LW
method required more than 500 steps.
\medskip

\noindent {\bf Second example:} Problem with noisy data. \\
Evolution of both iteration error and residual is shown in Figure~\ref{fig:calderon-noisy}.
The PLW method (GREEN) is compared with the LW method (BLUE) and with the
SD method (RED).
The stop criteria \eqref{def:discrep} is reached after 14 steps of the PLW iteration,
32 steps for the SD iteration, and 56 steps for the LW iteration.

\subsection{Second experiment: The semiconductor setup}

In this paragraph we consider the more realistic setup (in agreement with the
semiconductor models in Subsection~\ref{ssec:num-description}) with
$\partial\Omega_D \not\subseteq \partial\Omega$, and
$\Gamma_0 \neq \emptyset$, $\partial\Omega_N \neq \emptyset$.

In this experiment we have: (i) The voltage profile $U \in H^{1/2}(\partial\Omega_D)$
satisfies $U |_{\Gamma_1} = 0$; (ii) As in the previous experiment, the identification
of $a(x)$ is performed from a single measurement.
To the best of our knowledge, within this setting, Assumptions \textbf{A1} -- \textbf{A3}
where not yet established for the operator $F$ in \eqref{eq:ip-abstract} and its discretizations.
Therefore, although the operator $F$ is continuous \cite{BELM04}, it is still unclear
whether the analytical convergence results of the previous sections hold here.

The setup of the numerical experiments presented in this section is the following:

\smallskip\noindent $\bullet$ The elements listed below are the same as in the
previous experiment:

--- The domain $\Omega \subset \mathbb R^2$;

--- The parameter space $H^1(\Omega)$ and the function $a^\star(x)$ to be identified;

--- The computation of the ``exact data'' $y$ in \eqref{eq:ip-abstract};

--- The choice for the TCC constant $\eta$ in \eqref{eq:tcc} and for
    $\tau$ in \eqref{def:tau-a-b.t};

--- The level $\delta$ of artificially introduced noise;

--- The procedure to generate the noisy data $y^\delta$;

\smallskip\noindent $\bullet$ The boundary parts mentioned in Subsection~%
\ref{ssec:num-description} are defined by $\partial\Omega_D := \Gamma_0 \cup \Gamma_1$,
$\Gamma_1 := \{ (x,1) \, ;\ x \in (0,1) \}$,
$\Gamma_0 := \{ (x,0) \, ;\ x \in (0,1) \}$,
$\partial\Omega_N := \{ (0,y) \, ;\ y \in (0,1) \} \cup \{ (1,y) \, ;\ y \in (0,1) \}$. \\
%
%
(in Figure~\ref{fig:semicond-setup}~(a) and~(b), the boundary part $\Gamma_1$
corresponds to the lower left edge, while $\Gamma_0$ is the top right edge; the
origin is on the upper right corner).

\smallskip\noindent $\bullet$ The fixed Dirichlet input for the DtN map \eqref{eq:num-d2nB}
is the piecewise constant function $U: \partial\Omega_D \to \R$ is defined by
\ $U(x,0) := 1$, and $U(x,1) = 0$.
In Figure~\ref{fig:semicond-setup}~(a), $U(x)$ and the corresponding solution
$\hat u$ of (\ref{eq:num-d2nB}) are plotted.

\smallskip\noindent $\bullet$ The initial condition $a_0(x)$ used for all iterative
methods is shown in Figure~\ref{fig:semicond-setup}~(b) and is given by the solution of
the mixed boundary value problem
$$
\Delta a_0(x)      \ = \ 0 \, ,    \ \mbox{in } \Omega \, , \quad\quad
       a_0(x)      \ = \ U(x) \, , \ \mbox{at } \partial\Omega_D \, , \quad\quad
\nabla a_0\cdot\nu \ = \ 0  \, ,   \ \mbox{at } \partial\Omega_N \, ,
$$
analogously as in Remark~\ref{rem:init_guess}.

\smallskip\noindent $\bullet$ The computation of the iterative-step of the PLW method
is performed analogously as in Remark~\ref{rem:iter-step}, namely
$$
a_{k+1} \ := \ a_k \, - \, (1-\eta)
             \dfrac{\norm{F_0(a_k)}_{L^2(\Omega)}^2}{\norm{W_k}_{H^1(\Omega)}^2} \ W_k \, .
$$
where the Riesz vector $W_k \in H^1(\Omega)$ solves
$$
(I-\Delta) \, W_k  = \nabla\Psi \cdot \nabla u_k , \ \mbox{in } \Omega , \quad
              W_k  = 0 ,                           \ \mbox{at } \partial\Omega_D , \quad
\nabla W_k\cdot\nu = 0 ,                           \ \mbox{at } \partial\Omega_N ,
$$
and $\Psi$, $u_k$ solve
\begin{align*}
-\nabla \cdot (a_k(x) \, \nabla \Psi) = 0 ,             \ \mbox{in } \Omega , & \quad
                                 \Psi = F_0(a_k) , \ \mbox{at } \partial\Omega_{\Gamma_1} , \quad &
                   \nabla\Psi\cdot\nu = 0 ,             \ \mbox{at } \partial\Omega_N , \\
                   & \quad       \Psi = 0 , \hskip0.9cm \ \mbox{at } \partial\Omega_{\Gamma_0} , \quad & \\[1ex]
-\nabla \cdot (a_k(x) \, \nabla u_k) = 0 ,        \ \mbox{in } \Omega , & \quad
                                   u_k = U(x) , \,\ \ \mbox{at } \partial\Omega_D , \quad &
                    \nabla u_k\cdot\nu = 0 ,        \ \mbox{at } \partial\Omega_N .
\end{align*}

\medskip\noindent {\bf Example:} Problem with exact data. \\
Evolution of both iteration error and residual is shown in Figure~\ref{fig:semicond-exact}.
The PLW method (GREEN) is compared with the LW method (BLUE) and with the SD
method (RED).

\medskip\noindent {\bf Second example:} Problem with noisy data. \\
Evolution of both iteration error and residual is shown in Figure~\ref{fig:semicond-noisy}.
The PLW method (GREEN) is compared with the LW method (BLUE) and with the SD
method (RED).
The stop criteria \eqref{def:discrep} is reached after 9 steps of the PLW iteration,
22 steps for the SD iteration, and 153 steps for the LW iteration.

\begin{figure}[t!]
\centerline{ \includegraphics[width=9.2cm]{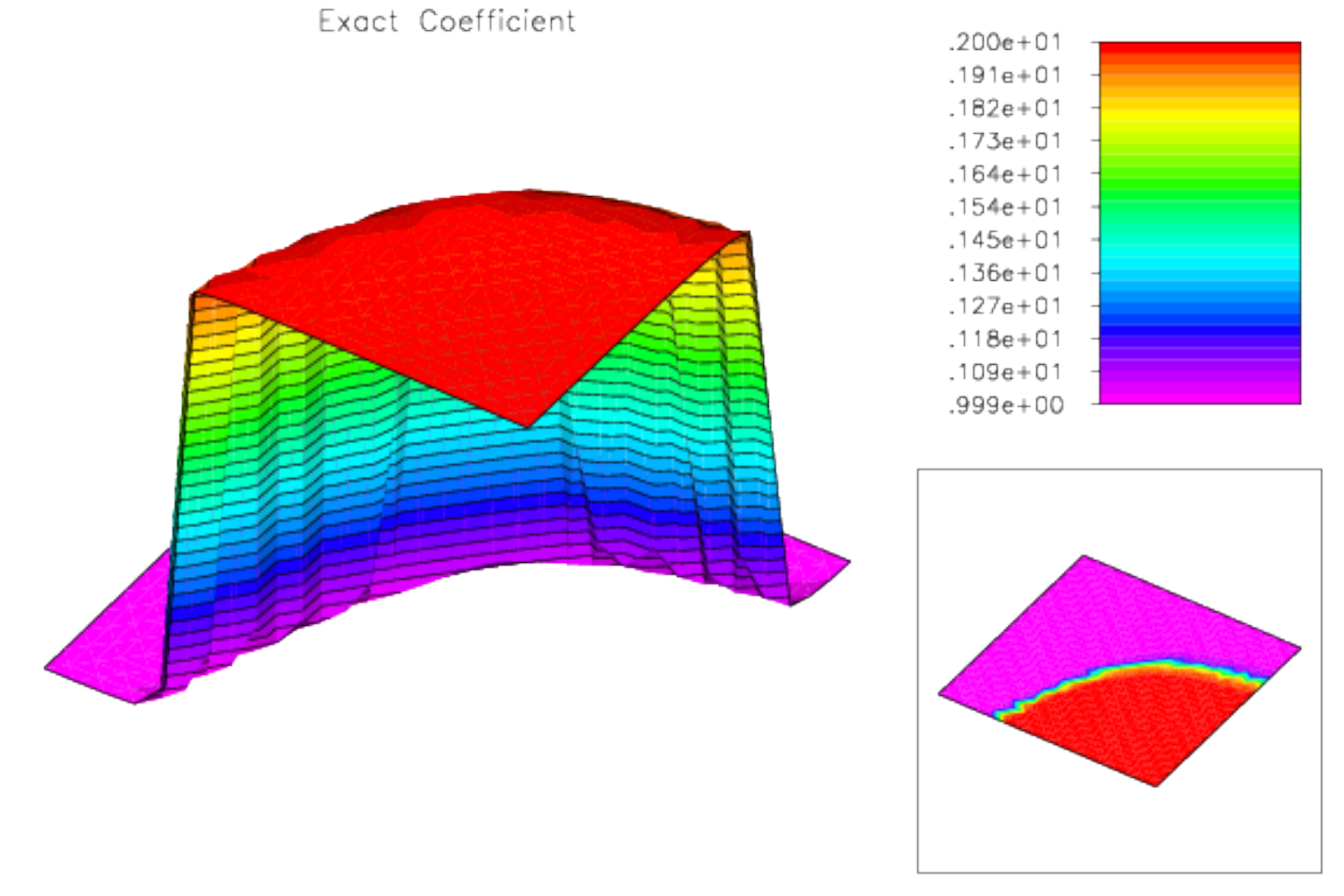} }
\centerline{ \includegraphics[width=9.2cm]{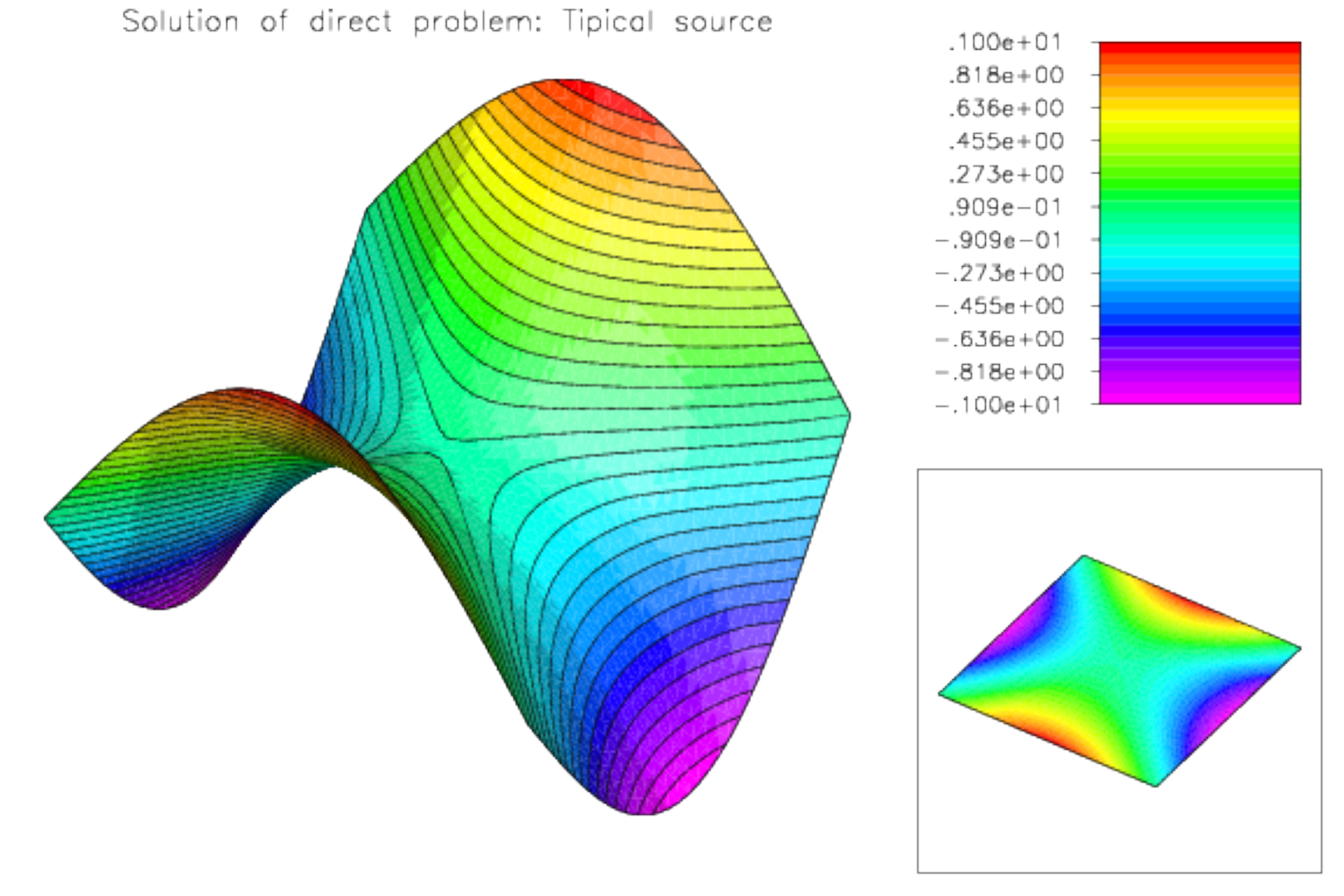} }
\centerline{ \includegraphics[width=9.2cm]{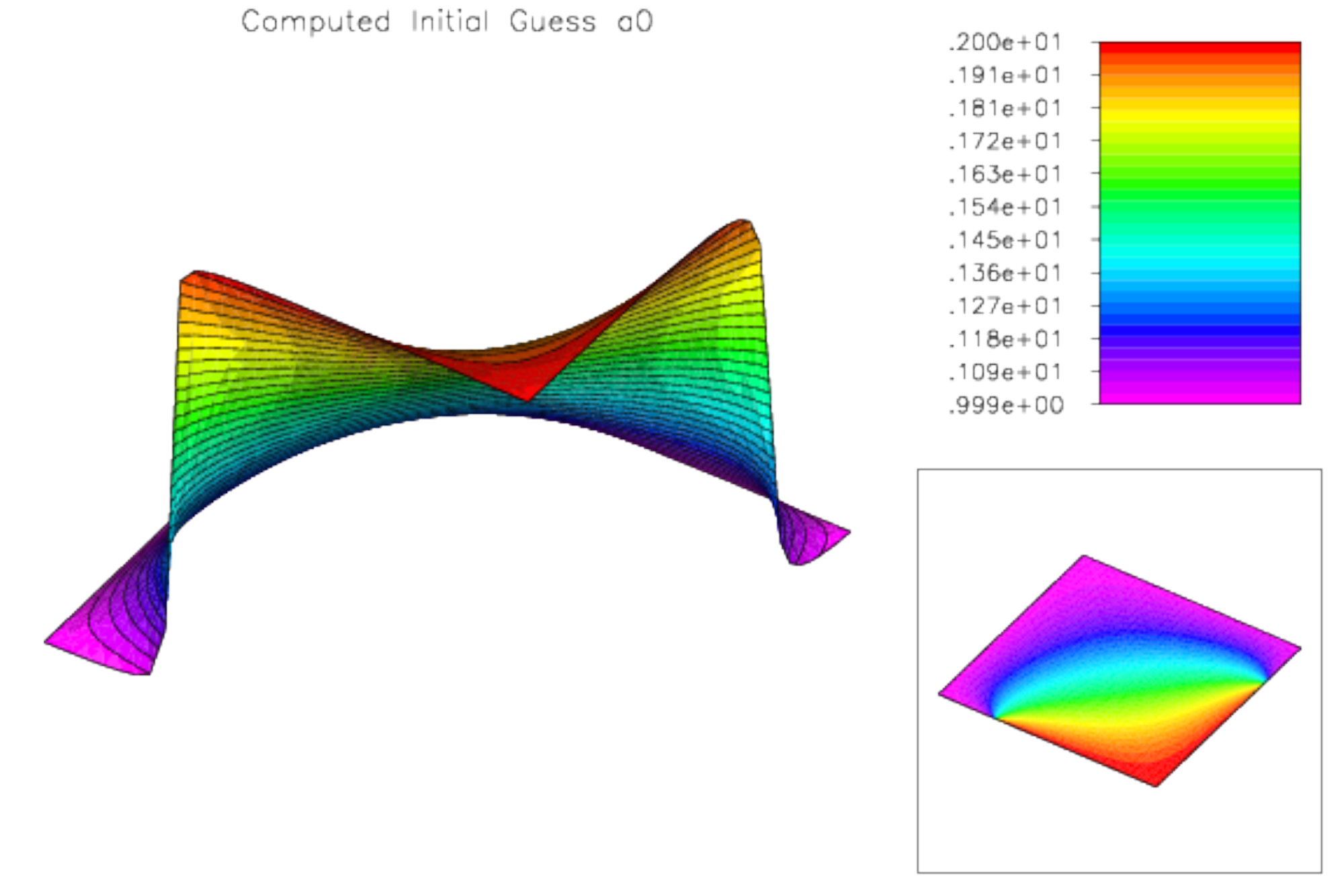} }
\caption{First experiment: setup of the problem.
{\bf Top:} The parameter $a^\star(x)$ to be identified;
{\bf Center:} Voltage source $U(x)$ (Dirichlet boundary condition at $\partial\Omega$
for the DtN map) and the corresponding solution $\hat u$ of \eqref{eq:num-d2nB};
{\bf Bottom:} Initial guess $a_0(x)$ for the iterative methods PLW, LW and SD.}
\label{fig:calderon-setup}
\end{figure}

\begin{figure}[t!]
\centerline{ \includegraphics[width=12cm,height=6cm]{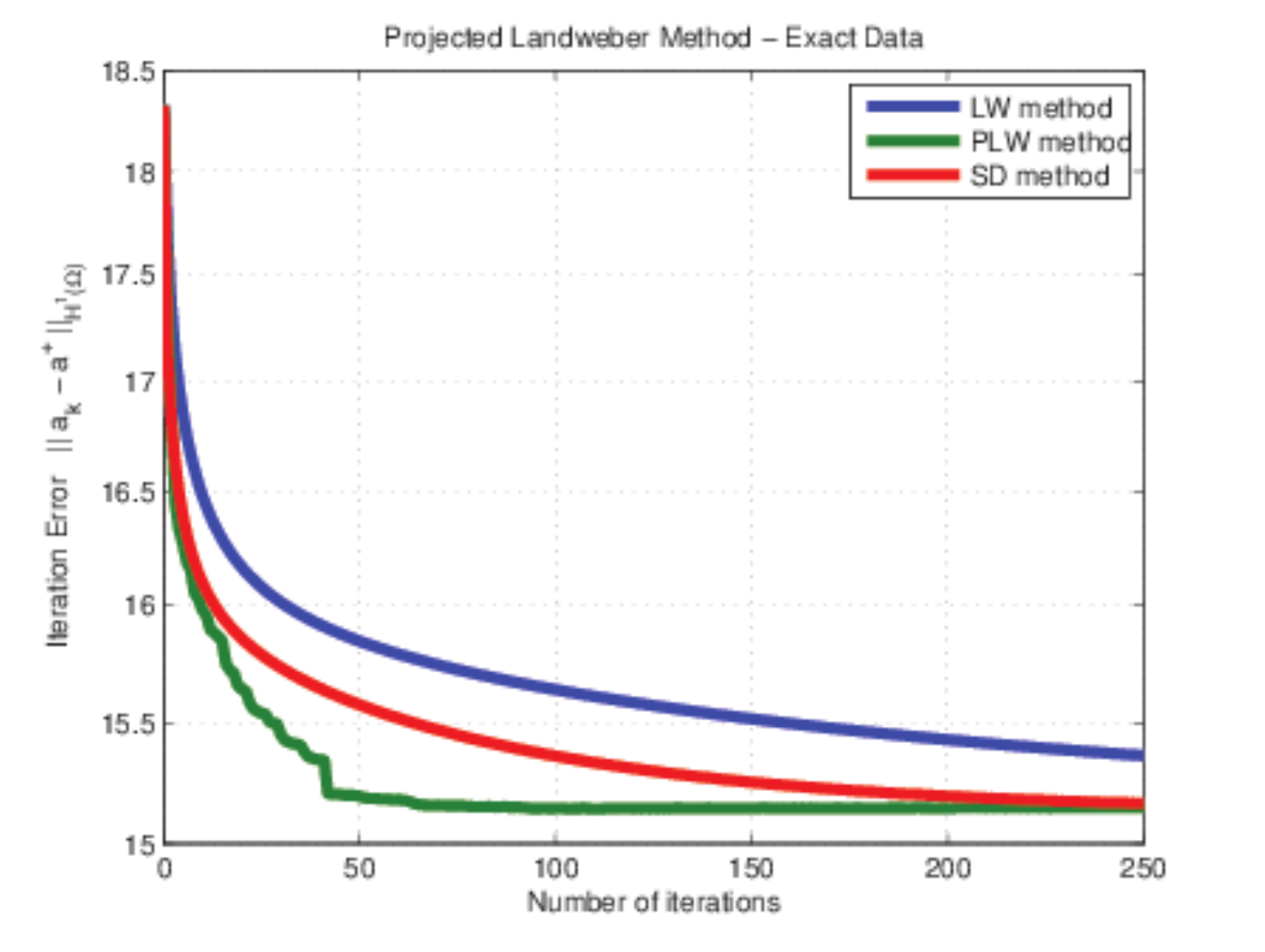} }
\centerline{ \includegraphics[width=12cm,height=6cm]{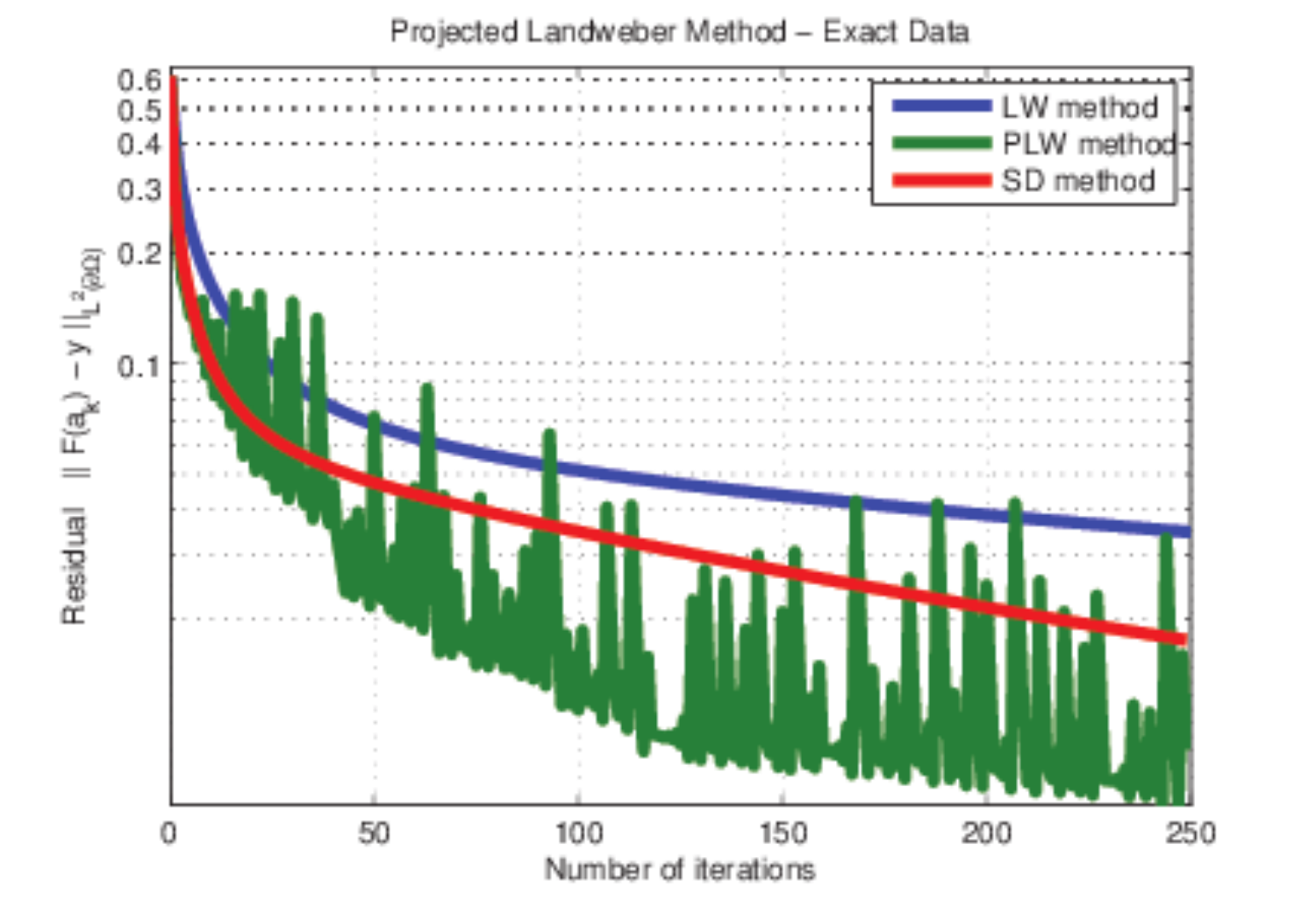} }
\centerline{ \includegraphics[width=12cm,height=6cm]{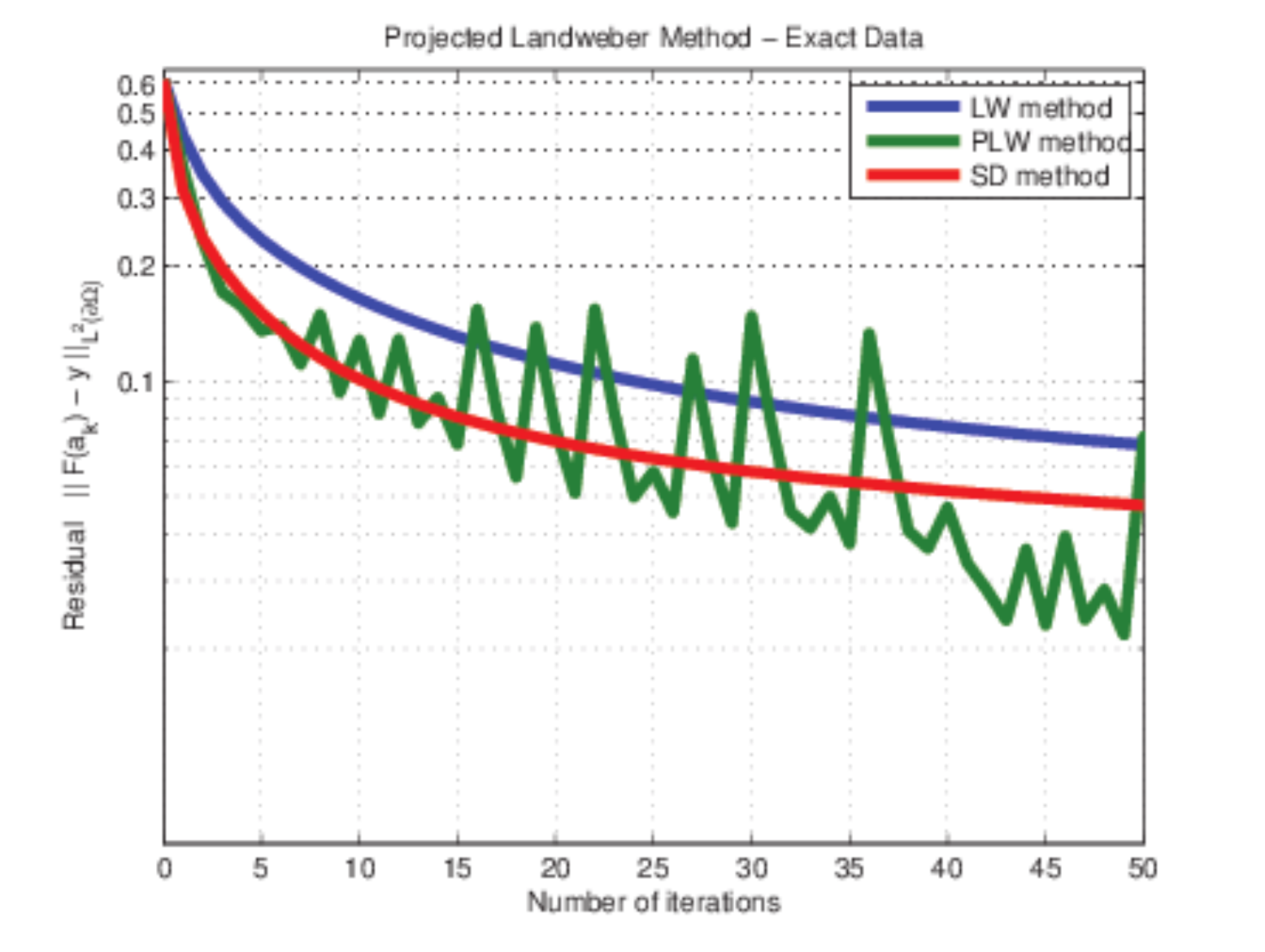} }
\caption{First experiment: example with exact data.
The PLW method (GREEN) is compared with the LW method (BLUE) and with the SD method (RED);
{\bf Top:} Iteration error $\norm{a_k - a^\star}_{H^1(\Omega)}$;
{\bf Middle:} Residual $\norm{F(a_k) - y}_{L^2(\partial\Omega)}$;
{\bf Bottom:} Residual, detail of the first 50 iterations.}
\label{fig:calderon-exact}
\end{figure}

\begin{figure}[t!]
\centerline{ \includegraphics[width=12cm]{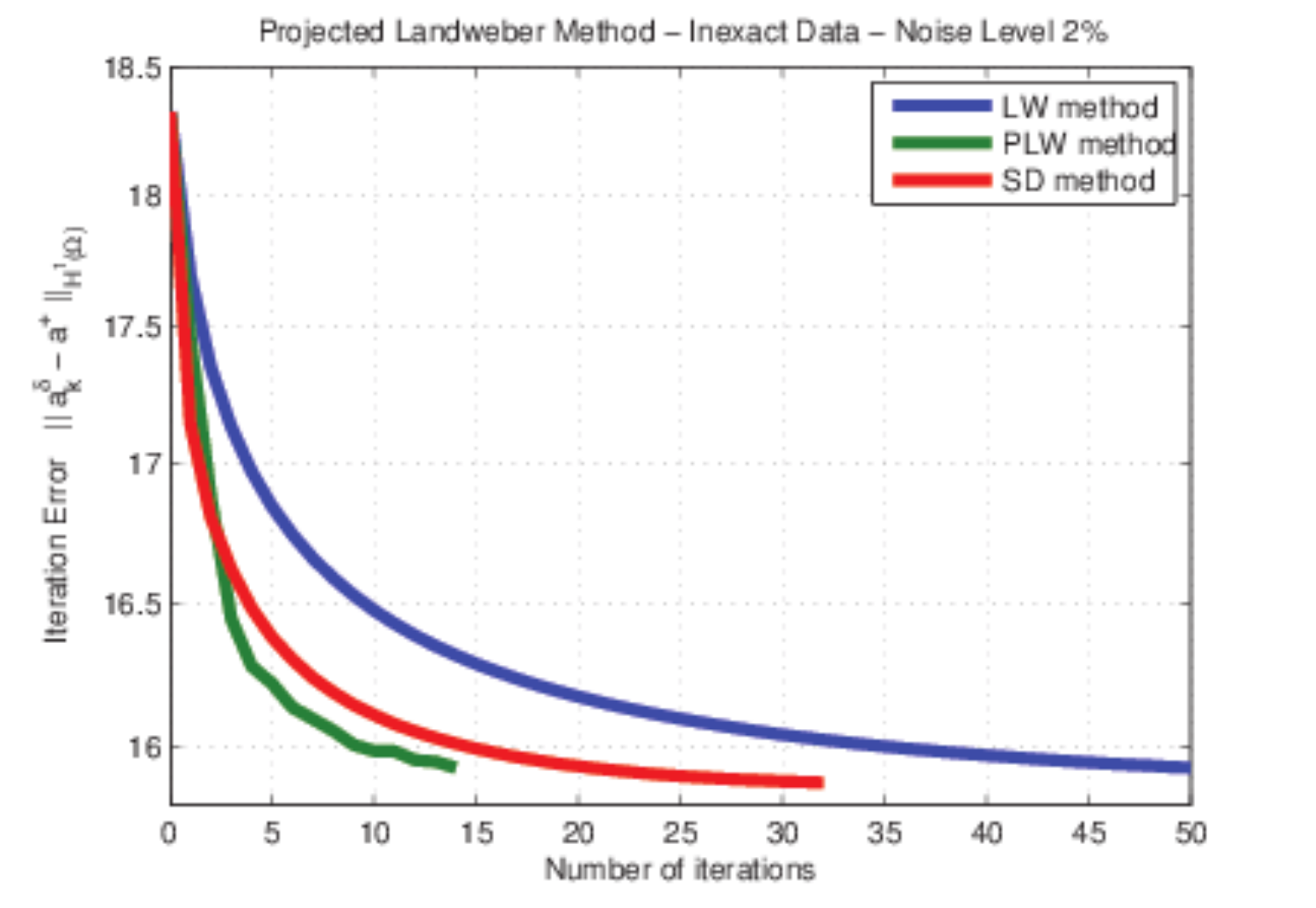} }
\centerline{ \includegraphics[width=12cm]{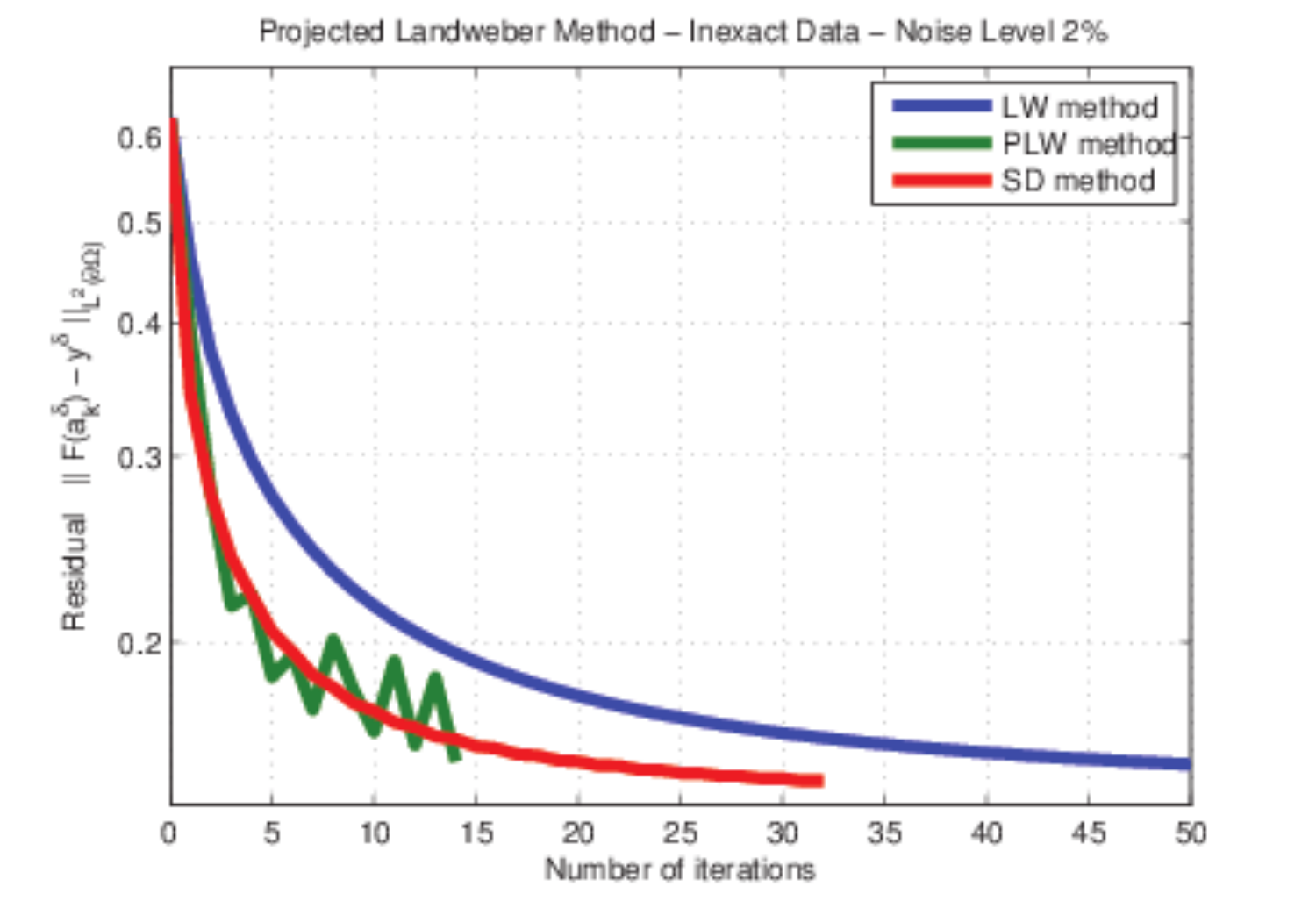} }
\caption{First experiment: example with noisy data.
The PLW method (GREEN) is compared with the LW method (BLUE) and with the SD method (RED);
{\bf Top:} Iteration error $\norm{a_k^\delta - a^\star}_{H^1(\Omega)}$;
{\bf Bottom:} Residual $\norm{F(a_k^\delta) - y^\delta}_{L^2(\partial\Omega)}$.}
\label{fig:calderon-noisy}
\end{figure}

\begin{figure}[t!]
\centerline{ \includegraphics[width=12.0cm]{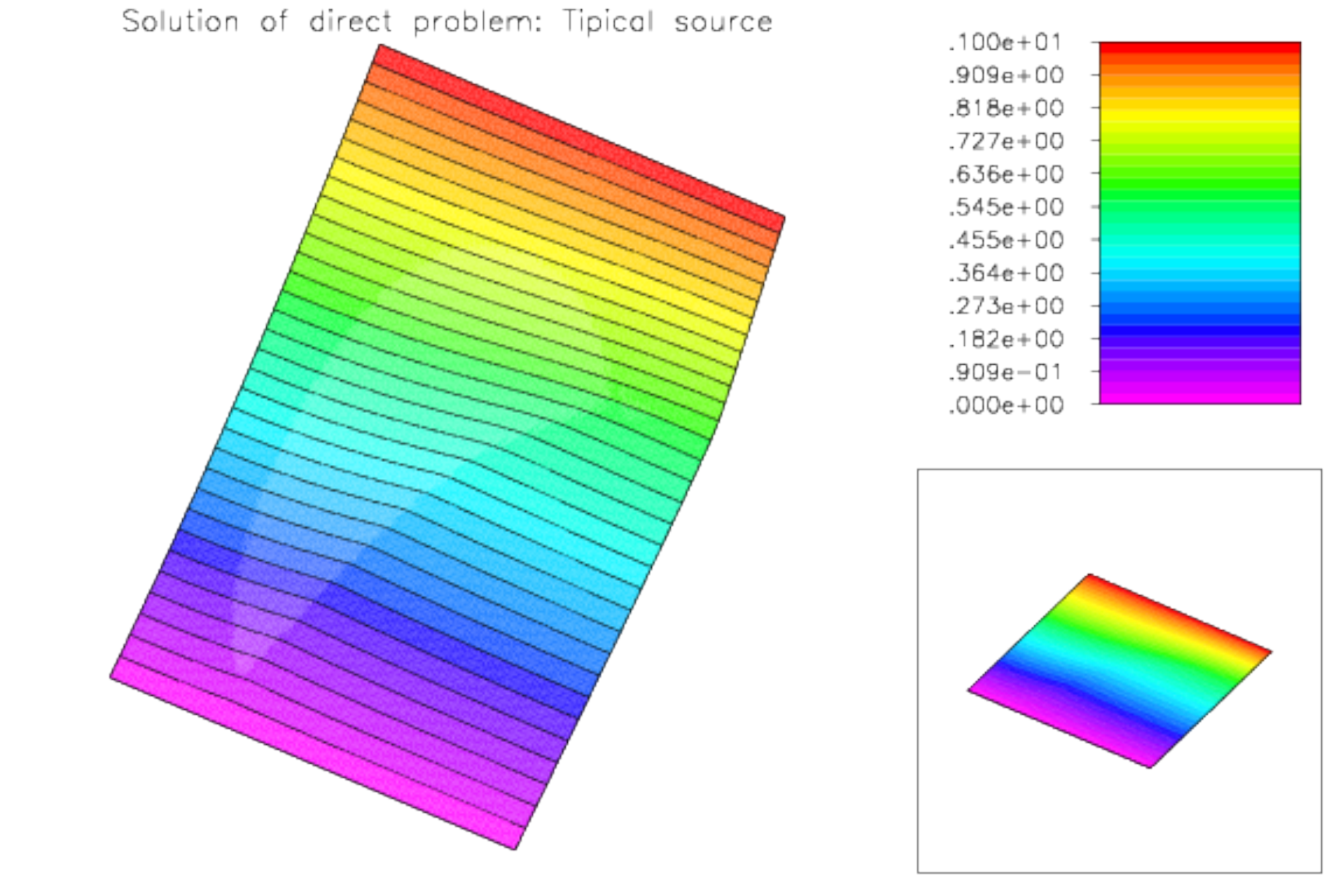} }
\centerline{ \includegraphics[width=12.0cm]{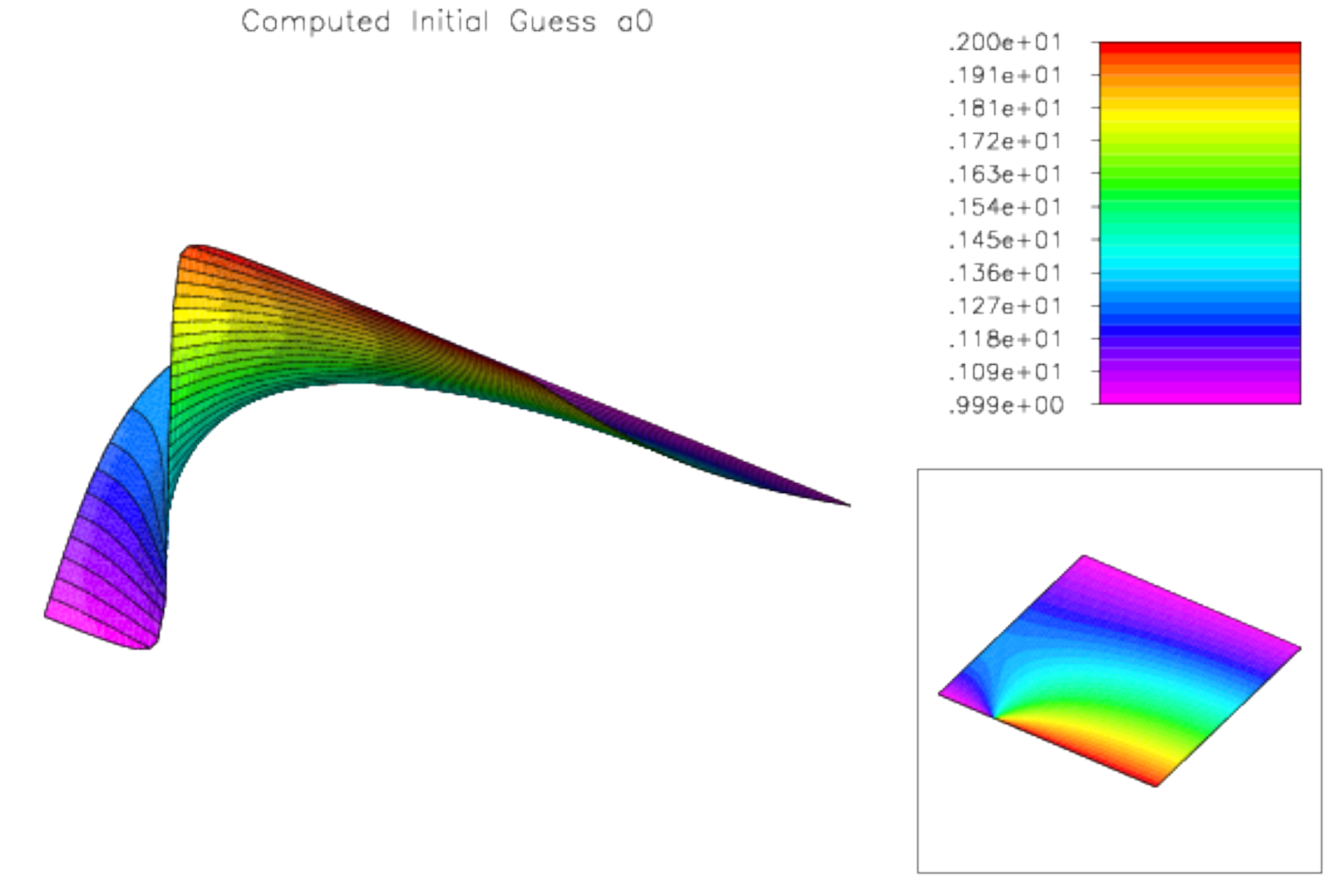} }
\caption{Second experiment: setup of the problem.
{\bf Top:} Voltage source $U(x)$ (Dirichlet boundary condition at $\partial\Omega_D$
for the DtN map) and the corresponding solution $\hat u$ of (\ref{eq:num-d2nB});
{\bf Bottom:} Initial guess $a_0 \in H^1(\Omega)$ satisfying $a_0(x) = U(x)$
at $\partial\Omega_D$ and $\nabla a_0(x)\cdot\nu(x) = 0$ at $\partial\Omega_N$.}
\label{fig:semicond-setup}
\end{figure}

\begin{figure}[t!]
\centerline{ \includegraphics[width=12cm,height=9cm]{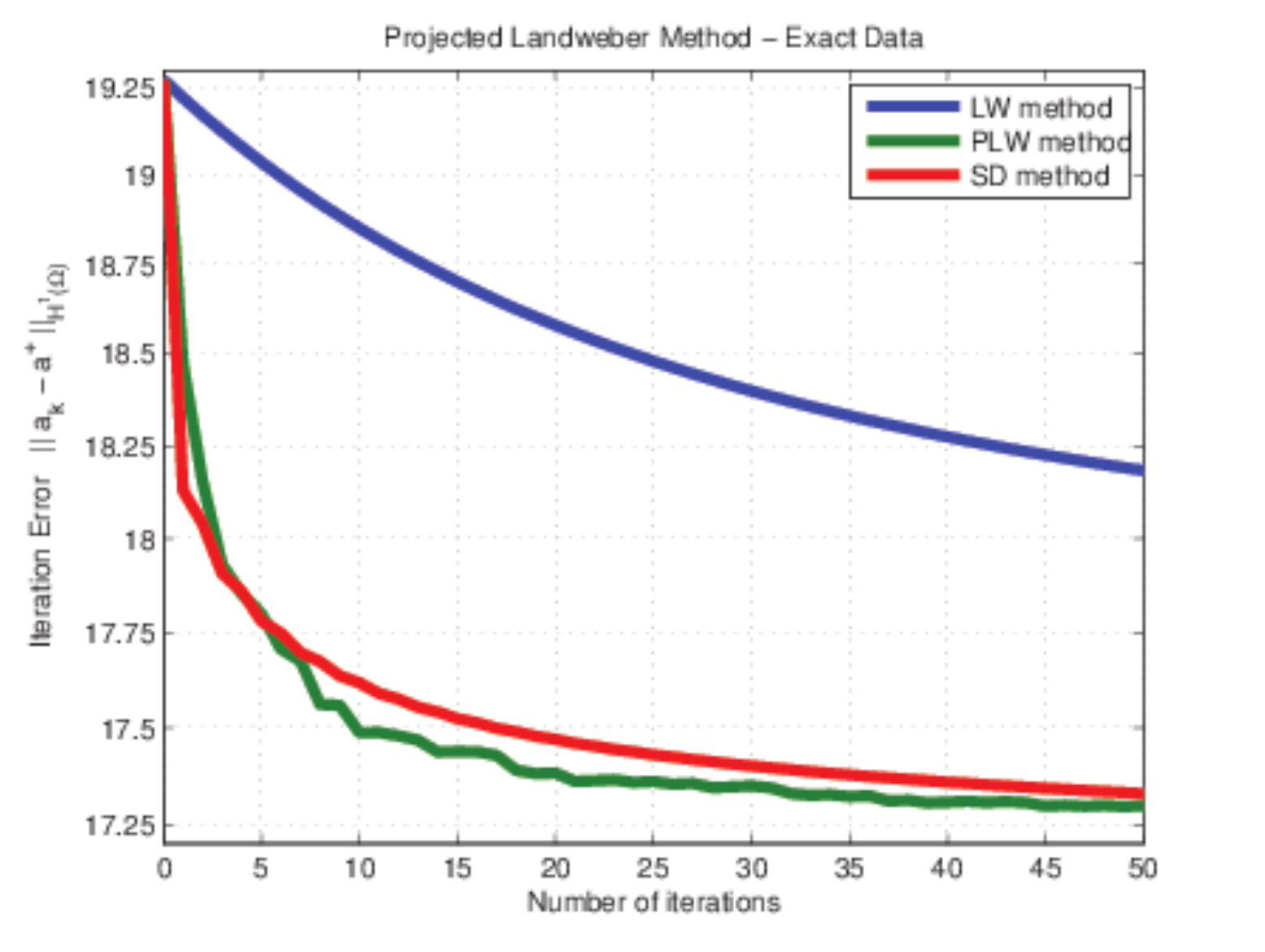} }
\centerline{ \includegraphics[width=12cm,height=9cm]{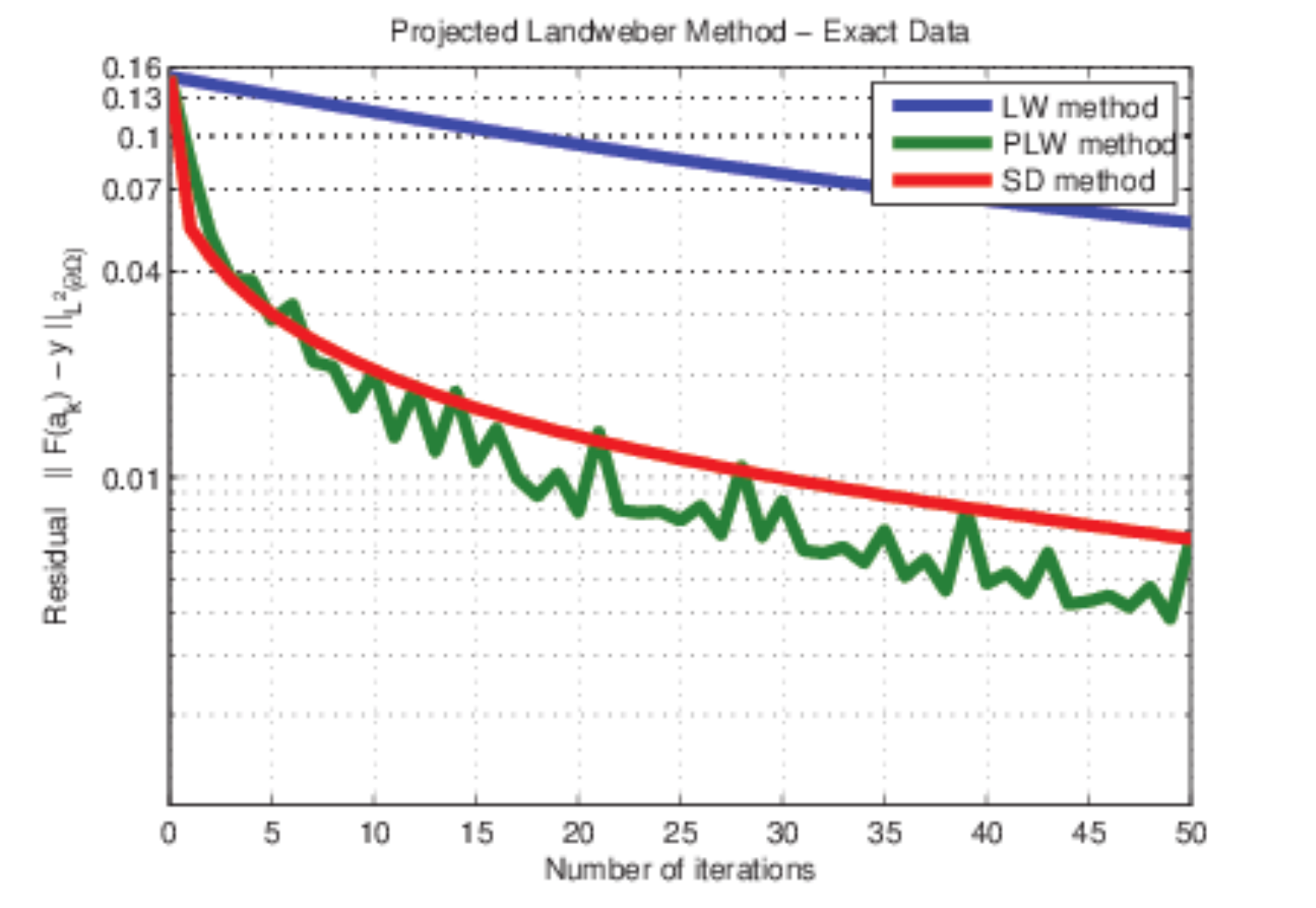} }
\caption{Second experiment: example with exact data.
The PLW method (GREEN) is compared with the LW method (BLUE) and with the SD method (RED);
{\bf Top:} Iteration error $\norm{a_k - a^\star}_{H^1(\Omega)}$;
{\bf Bottom:} Residual $\norm{F(a_k) - y}_{L^2(\Gamma_1)}$.}
\label{fig:semicond-exact}
\end{figure}

\begin{figure}[t!]
\centerline{ \includegraphics[width=12cm]{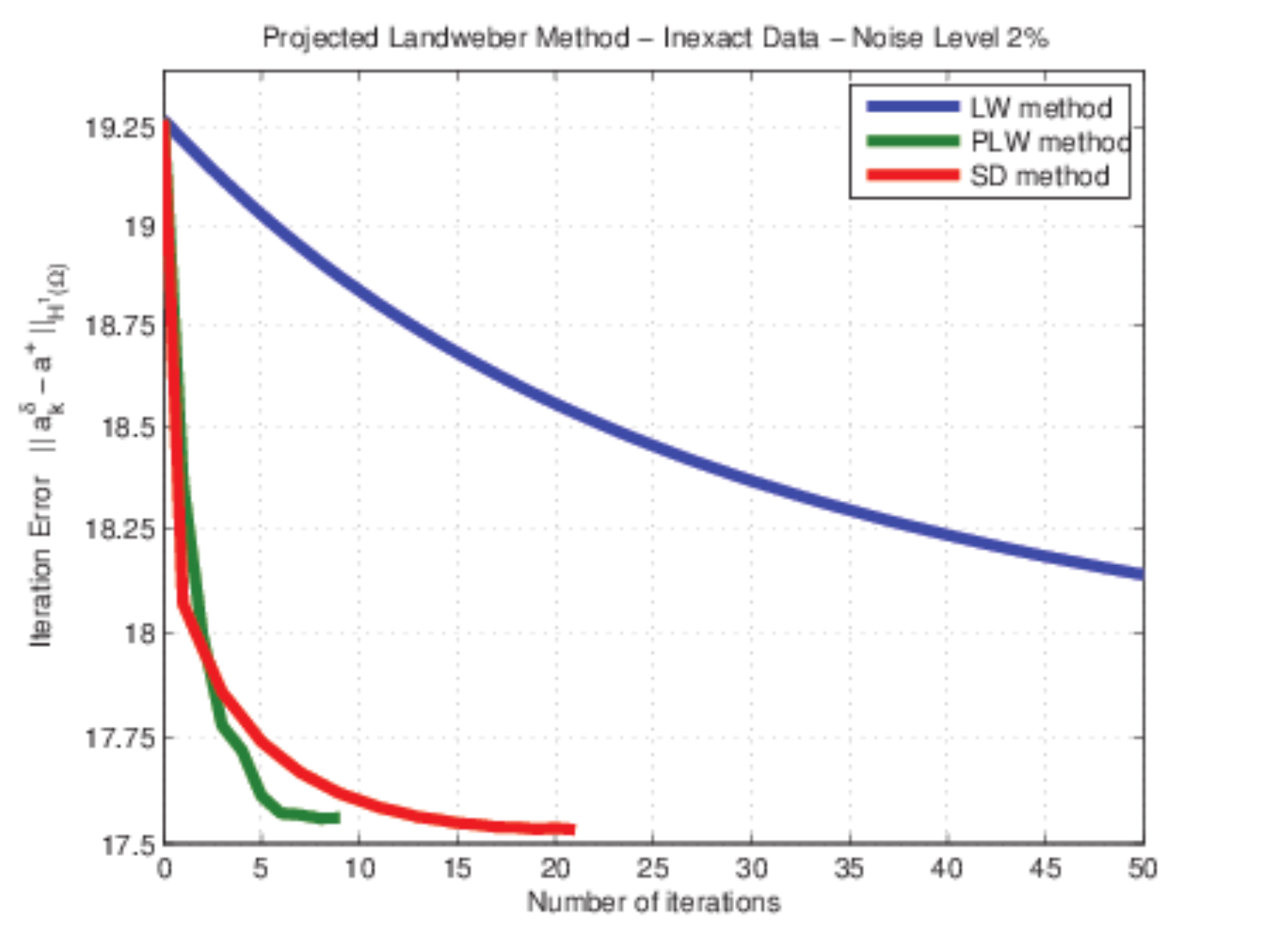} }
\centerline{ \includegraphics[width=12cm]{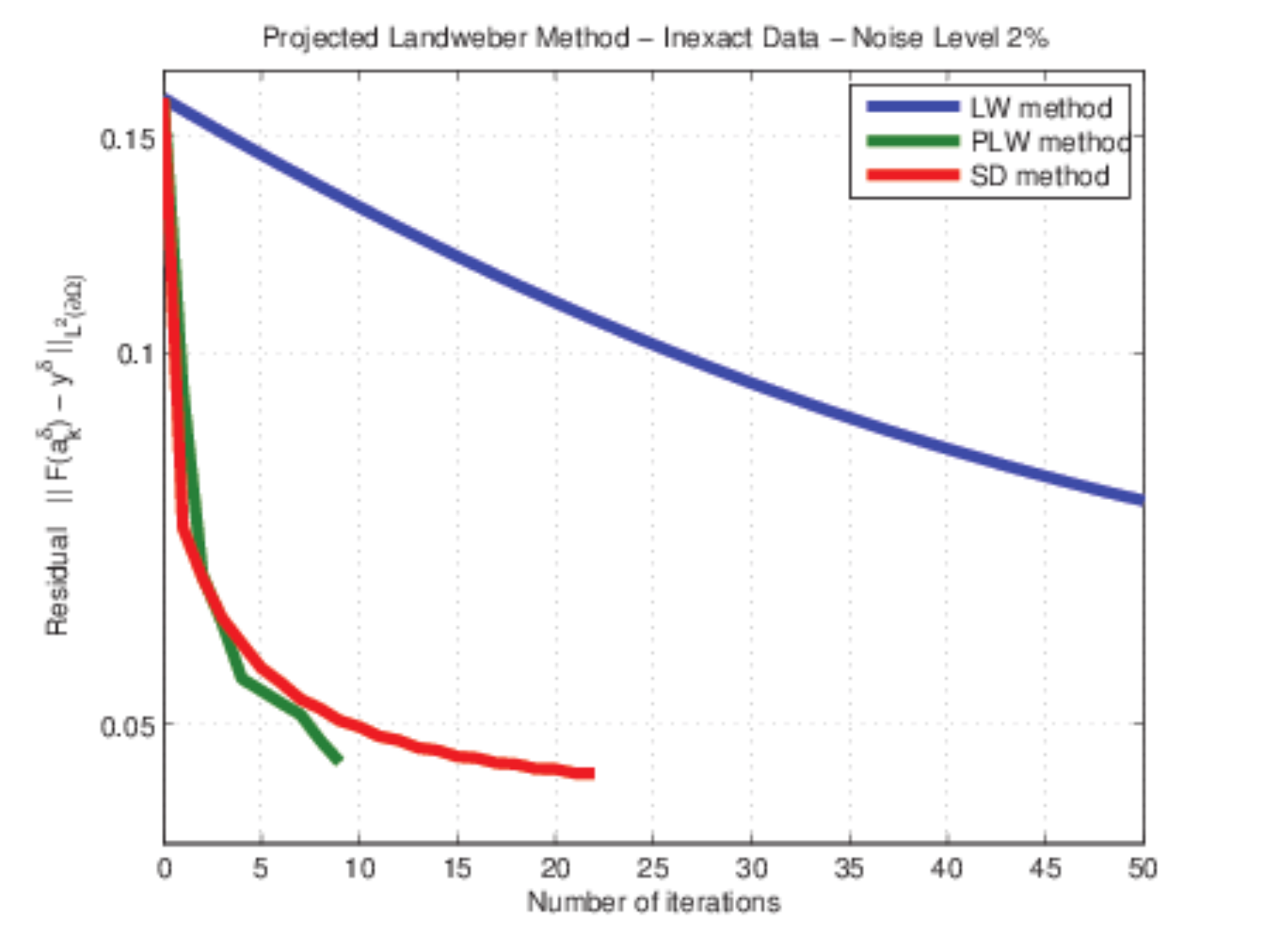} }
\caption{Second experiment: example with noisy data.
The PLW method (GREEN) is compared with the LW method (BLUE) and with the SD method (RED);
{\bf Top:} Iteration error $\norm{a_k^\delta - a^\star}_{H^1(\Omega)}$;
{\bf Bottom:} Residual $\norm{F(a_k^\delta) - y^\delta}_{L^2(\partial\Omega)}$.}
\label{fig:semicond-noisy}
\end{figure}

\section{Conclusions} \label{sec:conclusion}

In this work we use the TCC to devise a family of relaxed projection Landweber
methods for solving operator equation \eqref{eq:inv-probl}. The distinctive
features of this family of methods are:
\begin{itemize}
\item[$\bullet$] the basic method in this family (the PLW method) outperformed,
in our preliminary numerical experiments, the classical Landweber method as
well as the steepest descent method (with respect to both the computational
cost and the number of iterations);
\item[$\bullet$] the PLW method is convergent for the constant of the TCC in
a range {\em twice as large} as the one required for the convergence of Landweber
and other gradient type methods;
\item[$\bullet$] for noisy data, the iteration of the PLW method progresses
towards the solution set for residuals twice as small as the ones prescribed
by the discrepancy principle for Landweber \cite[Eq.~(11.10)]{EHN96} and
steepest descent \cite[Eq.~(2.4)]{Sch96} methods.
This follows from the fact that the constant prescribed by the discrepance
principle for our method and for Landweber/steepest-descent are, respectively
$$
\tau \ = \ \frac{1+\eta}{1-\eta} 
\quad \text{ and } \quad
\tau \ = \ 2 \frac{1+\eta}{1-2\eta};
$$
\item[$\bullet$] the proposed family of projec\-tion-type methods encompasses,
as particular cases, the Landweber method, the steepest descent method as well
as the minimal error method; thus, providing an unified framework for their
convergence analysis.
\end{itemize}

{\color{black}
In our numerical experiments, the \emph{residue} in the PLW method has
very strong oscillations for noisy data (Fig...) and for exact
data (Fig.)
Since this method iterations' aims to reduce the iteration error, a non-monotone
behavior of the residual is to be expected.
%
In ill-posed problem error and residual are poor correlated, which may
explain the large variations on the second one observed in our
experiments with the PLW.
%
Up to now it is not clear to us why this non-monotonicity happened to be
oscillatory in our experiments.
}

Although projection type methods for solving systems of \emph{linear} equations
dates back to \cite{Ci38,Kac37}, the use of these methods for ill-posed
equations is more recent, see, e.g, \cite{Nat86}.

A family of relaxed projection gradient-type methods for solving \emph{linear} ill-posed
operator equations was proposed in  \cite{MR75}.
In this work we extended to the non-linear case, under the TCC, the analysis of   \cite{MR75}.

\section*{Acknowledgments}

We thanks the Editor and the anonymous referees for the corrections and
suggestions, which improved the original version of this work.

A.L. acknowledges support from the Brazilian research agencies CAPES, CNPq
(grant 309767/2013-0), and from the AvH Foundation.
The work of B.F.S. was partially supported by CNPq (grants 474996/2013-1,
306247/ 2015-1) and FAPERJ (grants E-26/102.940/2011 and E-21/201.548/2014).

\bibliographystyle{abbrv}
\bibliography{projectedLW}

\end{document}